\documentclass[matprg]{svjour3}
\usepackage{graphics,graphicx,epsf,subfigure,epstopdf}
\usepackage{amsmath,amsxtra,amsfonts,amscd,amssymb,bm}
\usepackage{algorithm}
\usepackage{algpseudocode}
\usepackage[top=2.5cm,bottom=2.5cm,right=2.5cm,left=2.5cm]{geometry}
\usepackage[mathscr]{eucal}
\usepackage{color}
\numberwithin{equation}{section}
\graphicspath{ {./pic/} }
\usepackage{multirow}

\def\argmin{{\rm argmin}}

\def\eqnok#1{(\ref{#1})}

\newcommand{\tsum}{\textstyle\sum}
\newcommand{\bbe}{\mathbb{E}}

\newtheorem{thm}{Theorem}
\newtheorem{lem}[thm]{Lemma}
\newtheorem{cor}[thm]{Corollary}
\newtheorem{prop}[thm]{Proposition}

\newcommand{\bbr}{\Bbb{R}}
\newcommand{\nn}{\nonumber}

\def\vgap{\vspace*{.1in}}

\title{Dynamic Stochastic Approximation \\
for Multi-stage Stochastic Optimization}
\author{
     Guanghui Lan*
    \thanks{H. Milton Stewart School of Industrial and Systems
    Engineering, Georgia Institute of Technology, Atlanta, GA, 30332.
    (email: {\tt george.lan@isye.gatech.edu}).}
    \and
  Zhiqiang Zhou
    \thanks{H. Milton Stewart School of Industrial and Systems
    Engineering, Georgia Institute of Technology, Atlanta, GA, 30332.
    (email: {\tt zzhoubrian@gatech.edu}). }
}

\date{\today}

\begin{document}
\maketitle

\begin{abstract}
In this paper, we consider multi-stage stochastic optimization problems with convex objectives and conic constraints at each stage.
We present a new stochastic first-order method, namely the dynamic stochastic approximation (DSA) algorithm, for solving these
types of stochastic optimization problems. We show that DSA can achieve an optimal  ${\cal O}(1/\epsilon^4)$ rate of convergence
in terms of the total number of required scenarios when applied to a three-stage stochastic optimization problem.
We further show that this rate of convergence
can be improved to ${\cal O}(1/\epsilon^2)$ when the
objective function is strongly convex.  We also discuss variants of DSA for solving more general multi-stage
stochastic optimization problems with the number of stages $T > 3$. The developed DSA algorithms only need to go through the scenario tree once in order to
compute an $\epsilon$-solution of the multi-stage stochastic optimization problem. As a result,
the memory required by DSA only grows linearly with respect to the number of stages. To the best of our knowledge,
this is the first time that stochastic approximation type methods are generalized for
multi-stage stochastic optimization with $T \ge 3$.
\end{abstract}

\section{Introduction} \label{sec_intro}
Multi-stage stochastic optimization aims at optimal decision-making over multiple periods of time, where
the decision in the current period has to take into account what will happen in the future.
This type of decision-making is very important to a few applications areas,
including finance, logistics, robotics and clinic trials etc.
In this paper, we are interested in solving a class of multi-stage stochastic optimization problems given by
\begin{equation}\label{multi-stage prob1}
\begin{array}{l}
\min h^1(x^1,c^1) +\quad \bbe_{|\xi^1} \left[\min h^2(x^2,c^2) \quad+ \bbe_{|\xi^{[2]}}\left[ \ldots  +\quad \bbe_{|\xi^{[T-1]}}\left[\min \ h^T(x^T, c^T) \right] \right] \right]\\
  \text{s.t.} \ A^1 x^1 - b^1 \in K^1 \quad \quad \text{s.t.} \ A^2x^2 - b^2 -B^2x^1 \in K^2, \quad \quad \text{s.t.} \ A^Tx^T - b^T-B^Tx^{T-1} \in K^T, \\
   \quad\quad    x^1\in X^1,  \ \ \ \quad\quad \quad\quad \quad \ x^2\in X^2,\ \ \ \quad\quad\quad\quad\quad\quad\quad\quad \quad\ x^T\in X^T.
\end{array}
\end{equation}
Here $T$ denotes the number of stages, $h^t(\cdot,c^t)$ are relatively simple convex functions, $K^t$ are closed convex cones,
$X^t \subseteq \bbr^{n_t}$ are compact convex sets for some $n_t >0$,
$h^t: X^t \to \bbr$ are relatively simple convex functions, and
$A^t$ denote the linear mappings from $\bbr^{n_t}$ to $\bbr^{m_t}$ for some $m_t > 0$.
Moreover, $\xi^1 := (A^1,b^1,c^1)$ is a given deterministic vector,
$\xi^t := (A^t,b^t,B^t,c^t)$, $t=2, \ldots, T$, are the random vectors supported on $\Xi^t$
at stage $t$. Throughout this paper, we use $\xi^{[t]} := (\xi^1,\ldots, \xi^t)$ to denote the stochastic process up to time period $t$,
 and $\bbe_{|\xi^{[t]}}(\cdot) \equiv \bbe[\cdot | \xi^{[t]}]$ denote the expectation
 conditional on $\xi^{[t]}$. It is worth noting that $\xi^{[1]} = \xi^1$ and that $\bbe_{|\xi^1}[\cdot] \equiv \bbe_{|\xi^{[1]}}[\cdot] = \bbe[\cdot]$ since
 $\xi^1$ is deterministic.
By defining value functions, we can write
problem~\eqnok{multi-stage prob1} equivalently as
\begin{equation} \label{multi-stage prob}
\begin{array}{ll}
\min \ h^1(x^1,c^1)+ v^{2}(x^1, \xi^{[1]})  \\
\text{ s.t.} \ \ A^1 x^1 - b^1 \in K^1,\\
\quad \quad \quad x^1 \in X^1,
\end{array}
\end{equation}
where the value factions $v^t$ are recursively defined by
\begin{equation}\label{Defi_sto_V_m}
\begin{array}{lll}
 v^t(x^{t-1}, \xi^{[t-1]}) &:=& \bbe [V^t(x^{t-1},\xi^{[t]}) | \xi^{[t-1]}], \ \ t = 2, \ldots, T-1,\\
V^t(x^{t-1},\xi^{[t]}) &:= & \min \ h^t(x^t,c^t)+ v^{t+1}(x^t, \xi^{[t]})\\
 &&\ \text{  s.t.} \ \ A^tx^t - b^t -B^tx^{t-1} \in K^t,\\
 & & \quad \quad \quad x^t\in X^t,
 \end{array}
 \end{equation}
 and
\begin{equation}\label{Defi_sto_V1_m}
 \begin{array}{lll}
  v^T(x^{T-1}, \xi^{[T-1]}) &:=& \bbe [V^T(x^{T-1},\xi^{[T]})| \xi^{[T-1]}],\\
V^T(x^{T-1},\xi^{[T]}) &:=&  \min \ h^T(x^T,c^T)\\
  &&\text{ s.t.} \ \ A^T x^T - b^T - B^T x^{T-1} \in K^T, \\
  & &\quad \quad \quad x^T\in X^T.
\end{array}
\end{equation}

In particular, if $h^t$ are affine, $K^t = \{0\}$ and $X^t$ are polyhedral, then problem~\eqnok{multi-stage prob1}
reduces to the well-known multi-stage stochastic linear programming problem (see, e.g., \cite{BirLou97,ShDeRu09}).
The incorporation of the nonlinear (but convex) objective function $h^t(x^t,c^t)$  and conic constraints $A^tx^t - b^t -B^tx^{t-1} \in K^t$
allows us to model a much wider class of problems. Moreover, if $T = 2$, then problem~\eqnok{multi-stage prob1}
is often referred to as a two-stage (or static) stochastic programming problem.

In spite of its wide applicability, multi-stage stochastic optimization remains highly challenging to solve.
Many existing methods for multi-stage stochastic optimization are based on sample average approximation (see Nemirovski and Shapiro~\cite{ShaNem04} and
Shapiro~\cite{sha06}).
In this approach, one first generates a deterministic counterpart of \eqnok{multi-stage prob1} by replacing the expectations
with (conditional) sample averages. In particular, if the number of stages $T=3$, the total number of samples (a.k.a. scenarios)
cannot be smaller than ${\cal O}(1/\epsilon^4)$ in general. Once after a deterministic approximation of \eqnok{multi-stage prob1} is generated,
one can then develop decomposition methods to solve it to certain accuracy.
The most popular decomposition methods consist of stage-based and scenario-based decomposition method.
One widely-used stage-based method is the stochastic dual dynamic programming (SDDP) algorithm,
which is essentially an approximate cutting plane method, first presented by Pereira and Pinto \cite{pereira1991multi} and later studied
by Shapiro~\cite{Sha11}, Philpott et. al.~\cite{phil13-1}, Donohue and Birge~\cite{donohue2006abridged}, Hindsberger~\cite{hindsberger2014resa}, and Kozm\'{i}k and Morton~\cite{kozmik2015evaluating} etc.
This method has been shown to be effective for solving multi-stage stochastic optimization problems with a large number of stages, but
a small number of decision variables.
The progressive hedging algorithm by Rockafellar and Wets \cite{rockafellar1991scenarios} is a well-known  scenario-based decomposition method,
which basically applies an augmented Lagrangian method to penalize the violation of
the non-anticipativity constraints. Other interesting bundle type decomposition methods have also been
developed (see, e.g., \cite{HigSen91-1}).
These methods assume that the scenario tree has been generated and will go through the
scenario tree many times. Usually there are no performance guarantees
provided regarding their rate of convergence, i.e., the number of times one needs to go through the
scenario tree. In SDDP, one also needs to assume that random vectors
are stage-wise independent.

Recently, a different approach called stochastic approximation (SA) has attracted much attention
for solving  static stochastic optimization problems
given in the form of
\begin{equation} \label{sp}
\min_{x \in X} \left\{f(x) := \bbe_\xi[F(x, \xi)] \right\},
\end{equation}
where $X$ is a closed convex set, $\xi$ denotes the random vecctor and $F(\cdot, \xi)$
is a closed convex function.
Observe that when $T=2$,  problem~\eqnok{multi-stage prob1} can be cast in the form
of \eqnok{sp} and hence one can apply the aforementioned SA methods
to solve these two-stage stochastic optimization problems (see \cite{NJLS09-1,lns11}).
The basic SA algorithm, initially proposed by Robbins and Monro~\cite{RobMon51-1}, mimics the simple projected gradient descent method by replacing exact
gradient with its unbiased estimator.
Important improvements for the SA methods have been made by Nemirovski and Yudin~\cite{nemyud:83}
and later by Polayk and Juditsky \cite{pol90,pol92}.  During the past few years, significant progress has been made
in SA methods (e.g., \cite{NJLS09-1,Lan10-3,GhaLan12-2a,GhaLan13-1,GhaLan12,NO13-1,GhaLanZhang13-1,WangFangLiu16-1,WangMaGoldLiu17-1,DaiHePanBootsSong16-1}).
In particular, Nemirovski et. al. \cite{NJLS09-1} presented
a properly modified SA approach, namely, mirror descent SA for solving general nonsmooth convex SP problems.
Lan~\cite{Lan10-3} introduced an accelerated SA method, based on Nesterov's accelerated gradient method~\cite{Nest83-1},
for solving smooth, nonsmooth and stochastic optimization in a uniform manner.
Novel nonconvex SA methods and their accelerated versions have been studied in \cite{GhaLan12,GhaLanZhang13-1,GhaLan15-1,WangMaGoldLiu17-1}.
Some interesting progresses have also been made in \cite{WangFangLiu16-1,DaiHePanBootsSong16-1} for solving more complicated compositional
stochastic optimization problems. All these SA algorithms only need to access one single
$\xi_k$ at each iteration, and hence do not require much memory.
It has been shown in \cite{NJLS09-1,lns11} that SA methods can significantly outperform
the SAA approach for solving static (or two-stage) stochastic programming problems.
However, it remains unclear whether these SA methods can be generalized for multi-stage stochastic optimization
problems with $T \ge 3$.

In this paper, we attempt to shed some light on this problem by presenting
a dynamic stochastic approximation (DSA) method for multi-stage stochastic optimization.
The basic idea of the DSA method is to apply an inexact primal-dual SA method
for solving the $t$-th stage optimization problem to
compute an approximate stochastic subgradient for its associated value functions $v^t$. In the pursuit of this idea,
we manage to resolve the following
difficulties. First, the first-order information
for the value function $v^{t+1}$ used to solve the $t$-stage subproblem is not only
stochastic, but also biased. We need to control the bias associated with such first-order information.
In addition,
we need to develop a relationship between the primal-dual gap and the
error associated with approximate stochastic subgradients.
Second, in order to establish the convergence of stochastic optimization subroutines
for solving the $t$-stage problem, we need to guarantee that the
variance of approximate stochastic subgradients and hence the dual multipliers associated with the $(t+1)$-stage problem are bounded,
while no such results exist in the current SA literature. Third,
we need to make sure that the errors associated with approximate stochastic subgradients do not accumulate
quickly as the number of stages $T$ increases. By properly addressing these issues, we were able to show
that the DSA method can achieve an optimal  ${\cal O}(1/\epsilon^4)$ rate of convergence
in terms of the number of random samples when applied to a three-stage stochastic optimization problem. We further show that
this rate of convergence can be improved to ${\cal O}(1/\epsilon^2)$ when the objective function is strongly convex.
To the best of our knowledge, this is the first time that this improved ${\cal O}(1/\epsilon^2)$ complexity
has been obtained for solving three-stage problems under the strong convexity setting.
Even though the value functions for these problems are still convex (rather than strongly convex), by exploiting the structural information that the cost function $h^t$ at
each stage is strongly convex, our
algorithm can compute the approximate stochastic
subgradients more efficiently than the more general situation where the cost function $h^t$ at each stage is convex.
Moreover, we discuss variants of the DSA method which exhibit optimal rate of convergence for solving more general multi-stage
stochastic optimization problems with $T > 3$. The developed DSA algorithms only need to go through the scenario tree once in order to
compute an $\epsilon$-solution of the multi-stage stochastic optimization problem. As a result, the required memory for DSA
increases only linearly with respect to $T$.
To the best of our knowledge,
this is the first time that stochastic approximation type methods are generalized to and their complexities are established for
multi-stage stochastic optimization.
It should be also mentioned that although the main motivation and contribution of this paper
lie on the theoretical side of stochastic optimization,
the developed DSA algorithm provides an effective approach for solving stochastic optimization problems with a large number of decision
variables and a relatively smaller number of stages such as for those arising from hierarchical operations management and clinical trials.

This paper is organized as follows. In Section 2, we
introduce the basic scheme of the DSA algorithm and establish its main convergence properties
for solving three-stage stochastic optimization problems.
In Section 3, we show that the convergence rate of the DSA algorithm can be significantly improved under the strongly convex assumption on the objective function at each stage.
and we then develop variants of the DSA method for solving more general form of \eqref{multi-stage prob1} with $T > 3$ in Section 4.
Finally, some concluding remarks are made in Section~\ref{sec-remark}.


%

\vgap
\subsection{Notation and terminology}
For a closed convex set $X$, a function $\omega_X : X\mapsto R$ is called a distance generating function
with parameter $\alpha_X$, if $\omega_X$ is continuously differentiable and strongly convex with parameter $\alpha_X$ with respect to $\|\cdot\|$.
Therefore, we have
$$\langle y-x, \nabla\omega_X(y) - \nabla\omega_X(x) \rangle \geq \alpha_X \|y-x\|^2, \forall x,y\in X.$$
The prox-function associated with $\omega_X$ is given by
$$P_X(x,y) = \omega_X(y) - \omega_X(x) - \langle \nabla\omega_X(x), y-x \rangle, \forall x,y\in X. $$
It can be easily seen that
\begin{equation}\label{str_con_PD}
	P_X(x,y) \geq \tfrac{\alpha_X}{2}\|y-x\|^2,\  \forall x, y \in X.
\end{equation}
If $X$ is bounded, we define the diameter of the set $X$ as
\begin{equation}\label{OmegaX}
\Omega_X^2 := \max_{x,y\in X} P_X(x,y).
\end{equation}
For a given closed convex cone $K_*$, we choose the distance generating function $\omega_{K_*} (y)= \|y\|_2^2/2$.
For simplicity, we often skip the subscript of $\|\cdot\|_2$ whenever we apply it to an unbounded set (such as a cone).

For a given closed convex set $X \subseteq \bbr^n$ and a closed convex function $V: X \to \bbr$,
$g(x)$ is called an $\epsilon$-subgradient of $V$ at $x \in X$ if
\begin{equation} \label{def_eps_sub}
V(y) \geq V(x) + \langle g(x), y-x\rangle -\epsilon \ \ \forall y \in X.
\end{equation}
The collection of all such $\epsilon$-subgradients of $V$ at $x$ is called
the $\epsilon$-subdeifferential of $V$ at $x$,
 denoted by $\partial_\epsilon V(x)$.

 Assume that $V$ is Lipschitz continuous in an $\epsilon$-neighborhood of $X$, i.e.,
\begin{equation} \label{def_neighbor_X}
 |V(y) - V(x)| \le M_0 \|y - x\|, \ \forall x,y \in X_\epsilon := \{ p \in \bbr^n: p = r+x, x \in X, \|r\| \le \epsilon\}.
 \end{equation}
We can show that
 \begin{equation}\label{bnd_subgradient}
 \|g(x)\|_* \le M_0 + 1 \ \ \forall x \in X.
 \end{equation}
 Indeed, if $\|\cdot\| = \|\cdot\|_2$, the result follows immediately by setting $d = \epsilon g(x) / \|g(x)\|_2$ and $y = x+d$ in \eqnok{def_eps_sub}.
 Otherwise, we need to choose $d$ properly s.t. $\|d\| = \epsilon$ and $\langle g(x), d \rangle = \epsilon \|g(x)\|_*$.
 It should be noted, however, that if $V$ is Lipschitz continuous over $X$ (rather than $X_\epsilon$),
 then one cannot guarantee the boundedness of an $\epsilon$-subgradient of $V$.

\section{Three-stage problems with generally convex objectives} \label{sec_DSA}
Our goal in this section is to introduce the basic scheme of the DSA algorithm
and discuss its convergence properties. For the sake of simplicity, we
will focus on three-stage stochastic optimization problems with
simple convex objective functions in this section. Extensions to strongly convex cases and more general form of
multi-stage stochastic optimization problems will be studied in later sections.

\subsection{Value functions and stochastic $\epsilon$-subgradients}
Consider the following three-stage stochastic programming problem:
\begin{equation}\label{3-stage prob}
\begin{aligned}
\min \, & h^1(x^1,c^1)+ &\bbe_{|\xi^1} [\min \, & \ h^2(x^2,c^2)& +\bbe_{|\xi^{[2]}}[\min \, & \ h^3(x^3, c^3)]]\\
  \text{s.t.} &\ A^1 x^1 - b^1 \in K^1 & \text{s.t.} &\ A^2x^2 - b^2 -B^2x^1 \in K^2, &  \text{s.t.}&\ A^3x^3 - b^3-B^3x^2 \in K^3, \\
      & x^1\in X^1,  \ \ \ &  &\ x^2\in X^2,\ \ \ & &\ x^3\in X^3.
\end{aligned}
\end{equation}
As a particular example, if $h^t(x^t,c^t) = \langle c^t, x^t\rangle$, $K^t = \{0\}$ and $X^t$ are polyhedronal,
then problem \eqnok{3-stage prob} reduces to a well-known three-stage stochastic linear programming problem.

We can write problem~\eqnok{3-stage prob} in a more compact form by using value functions as discussed in
Section~\ref{sec_intro}.
More specifically, let $V^3(x^{2},\xi^3|\xi^2)$ be the stochastic
value function at the third stage and $v^3(x^2)$ be the corresponding expected value function
conditionally on $\xi^{[2]}$:
\begin{equation}\label{Defi_sto_V1}
 \begin{array}{lll}
V^3(x^{2},\xi^{[3]}) &:=&  \min \ h^3(x^3,c^3)\\
  &&\text{ s.t.} \ \ A^3x^3 - b^3 - B^3x^{2} \in K^3, \\
  & &\quad \quad \quad x^3\in X^3.\\
v^3(x^2,\xi^{[2]}) &:=& \bbe [V^3(x^{2},\xi^{[3]})|\xi^{[2]}].
\end{array}
\end{equation}
We can then define the stochastic value function $V^2(x^{1},\xi^2)$ and its
corresponding (expected) value function as
\begin{equation}\label{Defi_sto_V}
\begin{array}{lll}
V^2(x^{1},\xi^{[2]}) &:= & \min \ \left\{ h^2(x^2,c^2)+ v^{3}(x^2, \xi^{[2]}) \right\}\\
 &&\ \text{  s.t.} \ \ A^2x^2 - b^2 -B^2x^{1} \in K^2,\\
 & & \quad \quad \quad x^2\in X^2.\\
v^2(x^1,\xi^1) &:=& \bbe [V^2(x^{1},\xi^{[2]}) |\xi^1] = \bbe [V^2(x^{1},\xi^2)] .
 \end{array}
 \end{equation}
Problem \eqref{3-stage prob} can then be formulated equivalently as
\begin{equation} \label{first-stage-cp}
\begin{array}{ll}
\min \ \left\{ h^1(x^1,c^1)+ v^{2}(x^1,\xi^1) \right\} \\
\text{ s.t.} \ \ A^1 x^1 - b^1 \in K^1,\\
\quad \quad \quad x^1 \in X^1.
\end{array}
\end{equation}
Throughout this paper, we assume that the expected value functions $v^2(x^1,\xi^1)$ and $v^3(x^2,\xi^{[2]})$, respectively,
are well-defined and finite-valued for a given $\xi^1$ and any $x^1 \in X^1$, and any $x^2 \in X^2, \xi^2 \in \Xi^2$ almost surely.
We observe that the assumption that the values functions are well-defined
holds under various regularity conditions (see Section 3.2 of \cite{ShDeRu14}
for a more detailed discussion).
It is also worth noting that in the above formulation, we
assume that the value functions $v^t$ depend on the immediately preceding
decisions $x^{t-1}$, rather than all earlier decisions $x^1, \ldots, x^{t-1}$ for the sake of convenience. In
the latter case, one can reformulate the
problems in the form of \eqnok{Defi_sto_V}
by introducing the so-called model state variables (Section 3.1.2 of \cite{ShDeRu14}).

In order to solve problem~\eqnok{first-stage-cp}, we need to
understand how to compute first-order information about the
value functions $v^2$ and $v^3$. Since both $v^2$ and $v^3$
are given in the form of (conditional) expectation, their exact first-order information
is hard to compute. We resort to the computation of a stochastic $\epsilon$-subgradient
of these value functions defined as follows.
\begin{definition}
$G(u,\xi^{[t]})$ is called a stochastic $\epsilon$-subgradient of the value function $v^t(u, \xi^{[t-1]}) = \bbe[V^t(u, \xi^{[t]})|\xi^{[t-1]}]$
if $G(u,\xi^{[t]})$ is an unbiased estimator of an $\epsilon$-subgradient of $v^t(u, \xi^{[t-1]})$ with respect to $u$, i.e.,
\begin{equation}\label{def_g}
\bbe [G(u,\xi) | \xi^{[t-1]}] = g(u, \xi^{[t-1]}) \ \ \mbox{and} \ \ g(u,\xi^{[t-1]}) \in \partial_\epsilon v^t(u, \xi^{[t-1]}).
\end{equation}
\end{definition}

To compute a stochastic $\epsilon$-subgradient of
$v^2$  (resp., $v^3$), we have to compute an approximate subgradient
of the corresponding stochastic value function $V^2(x^1, \xi^{[2]})$ (resp., $V^3(x^2, \xi^{[3])}$).
To this end, we
further assume that strong Lagrange duality holds
 for the optimization problems defined in \eqnok{Defi_sto_V} (resp.,\eqnok{Defi_sto_V1}) almost surely. In other words, these problems can be formulated
 as saddle point problems:
\begin{align}
V^2(x^1, \xi^{[2]})&=\max_{y^2\in K^2_*}\min_{x^2\in X^2}\langle b^2+B^2x^1-A^2x^2, y^2\rangle + h^2(x^2,c^2)+ v^3(x^2, \xi^{[2]}),\label{primaldual2}\\
V^3(x^2, \xi^{[3]})&=\max_{y^3\in K^3_*}\min_{x^3\in X^3}\langle b^3+B^3x^2-A^3x^3, y^3\rangle + h^3(x^3,c^3),\label{primaldual3}
\end{align}
where $K^2_*$ and $K^3_*$ are corresponding dual cones to $K^2$ and $K^3$, respectively.
One set of sufficient conditions to guarantee the equivalence between \eqnok{Defi_sto_V} (resp.,\eqnok{Defi_sto_V1}) and \eqnok{primaldual2} (resp., \eqnok{primaldual3})
is that \eqnok{Defi_sto_V} (resp.,\eqnok{Defi_sto_V1})  is solvable and the slater condition holds~\cite{Rocka70}.

Observe that in order to solve \eqnok{primaldual2} and \eqnok{primaldual3}, we need to solve
a more generic saddle point problem:
\begin{equation} \label{primaldual2.2}
V(u,\xi) \equiv V(u, (A, b, B, C)):= \max_{y\in K_*}\min_{x\in X}\langle b+Bu-Ax, y\rangle + h(x,c)+ \tilde v(x),
\end{equation}
where $A: \bbr^n \to m$ and $B: \bbr^{n_0} \to m$ denote the linear mappings.
For example, \eqnok{primaldual3} is a special case of  \eqnok{primaldual2.2} with
$u = x^2$, $y = y^3$, $K_* = K^3_*$, $b=b^3$, $B=B^3$, $A=A^3$, $h = h^3$ and $\tilde v = 0$.
It is worth noting that the first stage problem can also be viewed as a special case of \eqnok{primaldual2.2},
since \eqnok{first-stage-cp} is equivalent to
\begin{equation} \label{primaldual1}
\max_{y \in K_*^1} \min_{x^1 \in X^1} \ \left\{ \langle b^1-A^1x^1, y^1\rangle + h^1(x^1,c^1)+ v^{2}(x^1,\xi^1) \right\}.
\end{equation}

Let
\[(x_*,y_*) \in Z \equiv X \times K_*\]
be a pair of optimal solutions of the saddle point problem \eqref{primaldual2}, i.e.,
\begin{align}
V(u,\xi) &= \langle y_*,b+Bu-Ax_*\rangle+h(x_*,c)+ \tilde v(x_*)
= h(x_*,c)+ \tilde v(x_*), \label{strong_duality}
\end{align}
where the second identity follows from the complementary slackness of Lagrange duality.
Below we provide a different characterization of an $\epsilon$-subgradient of $V$ other than the one in \eqnok{def_eps_sub}.
\begin{lem}\label{epsi_subg}
Let $\bar z:= (\bar x, \bar y) \in Z$ and $u \in \bbr^{n_0}$ be given. If
\begin{equation}\label{goal_Q}
\begin{aligned}
Q(\bar z; x,y_*) &:=  \langle y_*, b+Bu-A\bar x\rangle + h(\bar x,c)+ \tilde v(\bar x) \\
& \quad -\langle \bar y, b+Bu-Ax\rangle - h(x,c)- \tilde v(x) \leq \epsilon,\ \forall x\in X,
\end{aligned}
\end{equation}
then $B^T\bar y$ is an $\epsilon$-subgradient of $V(u, \xi)$ at $u$.
\end{lem}
\begin{proof}
For simplicity, let us denote $V(u) \equiv V(u, \xi)$.
For any $u_1 \in {\rm dom} V$, we denote $(x_1^*,y_1^*)$ as
a pair of primal-dual solution of \eqnok{primaldual2.2} (with $u = u_1$). Hence,
\begin{equation} \label{def_Vu1}
V(u_1)= \langle y_1^*, b+Bu_1-Ax_1^*\rangle + h(x_1^*,c)+ \tilde v(x_1^*).
\end{equation}
It follows from the definition of $V$ in \eqnok{primaldual2.2} and \eqref{goal_Q} that
\begin{equation}\label{vb_prop6}
\begin{aligned}
V(u) &=  \langle y_*, b+Bu-Ax_* \rangle+h(x_*,c)+\tilde v(x_*) \\
&\leq  \langle y_*, b+Bu-A\bar x \rangle+h(\bar x,c)+ \tilde v(\bar x)\\
&\leq  \langle \bar y, b+Bu-Ax_1^*\rangle+h(x_1^*,c)+\tilde v(x_1^*)+\epsilon.\\
\end{aligned}
\end{equation}
Observe that
\begin{align*}
\langle \bar y, b+Bu-Ax_1^*\rangle &= \langle \bar y, B(u-u_1)\rangle +  \langle \bar y, b+Bu_1-Ax_1^*\rangle\\
&\le \langle \bar y, B(u-u_1)\rangle +  \langle y_1^*, b+Bu_1-Ax_1^*\rangle,
\end{align*}
where the last inequality follows from the assumption that $(x_1^*,y_1^*)$ is
a pair of optimal solution of \eqnok{primaldual2.2} with $u = u_1$.
Combining these two observations and using \eqnok{def_Vu1}, we have
\[
V(u) \le  \langle B^T\bar y,u-u_1\rangle + V(u_1)+\epsilon,
\]
which, in view of \eqnok{def_eps_sub}, implies that
$B^T\bar y$ is an $\epsilon$-subgradient of $V(u)$.
\end{proof}

\vgap

In view of Lemma~\ref{epsi_subg},
in order to compute a stochastic subgradient of $v^t(u, \xi^{[t-1]})=\bbe[V^t(u,\xi^{[t]})|\xi^{[t-1]}]$ at a given point $u$,
we can first generate a random realization $\xi^t$ conditionally on $\xi^{[t-1]}$ and then try to find a pair of solutions $(\bar x, \bar y)$
satisfying 
\begin{align*}
\langle y_*^t, b^t+B^tu-A^t \bar x\rangle + h(\bar x,c^t)+ v^{t+1}(\bar x, \xi^{[t]})
 -\langle \bar y, b^t+B^tu-A^tx\rangle - h(x,c^t)- v^{t+1}(x, \xi^{[t]}) \leq \epsilon,\ \forall x\in X,
\end{align*}
where $y_*^t \equiv y_*^t(\xi^{[t]})$ denotes the optimal solution for the $t$-th stage problem
associated with
the random realization $\xi^{[t]}$.
We will then use $B^T \bar y$ as a stochastic $\epsilon$-subgradient
of $v^t(u, \xi^{[t-1]})$ at $u$. However, the difficulty associated with this approach
exists in that the function $v^{t+1}(\bar x, \xi^{[t]})$ is also given in the form
of expectation.
We will explore this approach and discuss how to
address these issues in more details in the next subsection.

\subsection{The DSA algorithm}

Our goal in this subsection is to present the basic scheme of our dynamic stochastic
approximation algorithm applied to problem~\eqnok{first-stage-cp}.

Our algorithm relies on the following three key primal-dual steps, referred to as stochastic primal-dual
transformation (SPDT), applied to the generic saddle point problem in \eqnok{primaldual2.2}
at every stage.

\vgap

$(p_+, d_+, \tilde d) = {\rm SPDT}(p, d, d_{\_}, \tilde v', u, \xi, h, X, K_*, \theta, \tau, \eta)$:
\begin{align}
\tilde d &= \theta (d - d_{\_})+ d. \label{def_dual_extra}\\
p_+ &= \argmin_{x \in X} \langle b+Bu-Ax, \tilde d  \rangle + h(x,c) +\langle \tilde v',x \rangle + \tau P_X(p,x). \label{def_primal_proj}\\
d_+ &=  \argmin_{y\in K_*} \langle -b-Bu+A p_+, y\rangle + \tfrac{\eta}{2} \|y - d\|^2. \label{def_dual_proj}
\end{align}

In the above primal-dual tranformation, the input $(p,d, d_{\_})$ denotes the current primal solution, dual solution, and
the previous dual solution, respectively. Moreover, the input $\tilde v'$ denotes a stochastic
$\epsilon$-subgradient for $\tilde v$ at the current search point $p$. The parameters $(u, \xi, h, X, K_*)$
describes the problem in \eqnok{primaldual2.2} and $(\theta, \tau, \eta)$ are certain algorithmic parameters
to be specified. Given these input parameters, the relation in \eqnok{def_dual_extra} defines a dual extrapolation (or prediction) step
to estimate the dual variable $\tilde d$ for the next iterate. Based on this estimate, \eqnok{def_primal_proj}
performs a primal prox-mapping to compute $p_+$, and then \eqnok{def_dual_proj} updates in the dual
space to compute $d_+$ by using the updated $p_+$. We assume that the above SPDT operator
can be performed very fast or even has explicit expressions.
The primal-dual transformation is closely related to the
alternating direction method of multipliers and was first formally presented by  Chambolle and Pork in \cite{ChamPoc11-1}
for solving saddle point problems.
Its inherent relationship with Nesterov's acceleration has also been recently studied by Lan and Zhou~\cite{lan2015optimal}.

Observe that by the optimality conditions of \eqnok{def_primal_proj} and \eqnok{def_dual_proj} (see, e.g., Lemma 1 of \cite{LaLuMo11-1}),
the solution $(p_+, d_+, \tilde d)$ obtained from SPDT satisfies
\begin{align}
\langle -A(p_+ - x), \tilde d \rangle + h(p_+,c)& -h(x,c) + \langle \tilde v', p_+ -x\rangle \nonumber\\
    &\leq \tau [P_X(p,x)-P_X(p_+,x) - P_X(p,p_+)], \forall x \in X, \label{sto_x_opt0}\\
     \langle -b-Bu+Ap_+, d_+ - y\rangle &\leq \tfrac{\eta}{2} [\|d-y\|^2-\|d_+-y\|^2-\|d_{+}-d\|^2], \forall y \in K_*.\label{sto_y_opt0}
\end{align}

In order to solve problem~\eqnok{first-stage-cp}, we will combine the above primal-dual transformation
applied to all the three stages,
the scenario generation for the random variables $\xi^2$ and $\xi^3$
in the second and third stage, and certain averaging steps in both the primal and dual spaces.
We are now ready to describe the basic scheme of the DSA algorithm.
\begin{algorithm}[H] \label{basic_DSA}
\caption{The basic DSA algorithm for three-stage problems}
\begin{algorithmic}
\State {\bf Input:} initial points $(z_0^1,z_0^2, z_0^3)$.
\State $\xi^1 = (A^1, b^1, c^1)$.
\For {$i =1,2,\ldots,N_1$}
\State Generate  a random realization of $\xi_i^2 = (A_i^2,B_i^2,b_i^2,c_i^2)$.
\For {$j = 1,2,\ldots, N_2$}
\State Generate  a random realization of $\xi_j^3 = (A_j^3,B_j^3,b_j^3,c_j^3)$ (conditional on $\xi_i^2$).
\For {$k =1, 2, \ldots, N_3}$
\State $(x_k^3, y_k^3, \tilde y_k^3) = {\rm SPDT}(x_{k-1}^3, y_{k-1}^3, y_{k-2}^3, 0, x_{j-1}^2, \xi_j^3, h^3, X^3, K_*^3, \theta^3_k, \tau^3_k, \eta^3_k)$.
\EndFor
\State $(\bar x^3_j, \bar y^3_j) =  \tsum_{k=1}^{N_3} w^3_k (x_k^3, y_k^3) / \tsum_{k=1}^{N_3} w^3_k $.
\State $(x_j^2, y_j^2, \tilde y_j^2) = {\rm SPDT}(x_{j-1}^2, y_{j-1}^2, y_{j-2}^2, (B_j^3)^T \bar y_j^3, x_{i-1}^1, \xi_i^2, h^2, X^2, K_*^2, \theta^2_j, \tau^2_j, \eta^2_j)$.
\EndFor
\State $(\bar x^2_i, \bar y^2_i) =  \tsum_{j=1}^{N_2} w^2_j (x_j^2, y_j^2) / \tsum_{j=1}^{N_2} w^2_j$.
\State $(x_i^1, y_i^1, \tilde y_i^1) = {\rm SPDT}(x_{i-1}^1, y_{i-1}^1, y_{i-2}^1, (B_i^2)^T \bar y_i^2, 0, \xi^1, h^1, X^1, K_*^1, \theta^1_i, \tau^1_i, \eta^1_i)$.
\EndFor
\State {\bf Output:} $(\bar x^1, \bar y^1) = \tsum_{i=1}^{N_1} w_i^1 (x_i^1, y_i^1) / \tsum_{i=1}^{N_1} w_i^1$.
\end{algorithmic}
\end{algorithm}
This algorithm consists of three loops. The innermost (third) loop runs $N_3$ steps of SPDT in order to compute an approximate stochastic
subgradient ($(B_j^3)^T \bar y_j^3$) of the value function $v^3$ of the third stage.
The second loop consists of $N_2$ SPDTs applied to the saddle point formulation of the second-stage problem,
which requires the output from the third loop. The outer loop applies $N_1$ SPDTs to the saddle point formulation
of the first-stage optimization problem in \eqnok{first-stage-cp}, using the approximate stochastic  subgradients ( $(B_i^2)^T \bar y_i^2$)
for $v^2$ computed by the second loop. In this algorithm, we need to generate $N_1$ and $N_1 \times N_2$ realizations
for the random vectors $\xi^2$ and $\xi^3$, respectively. Observe that the DSA algorithm
described above is conceptual only since we have not specified any algorithmic parameters yet.
We will come back to this issue after establishing some general convergence properties
about this method in the next two subsections.


\subsection{Basic tools: inexact primal-dual stochastic approximation}
In this subsection, we provide some basic tools for the convergence analysis of the DSA method. 
In particular, we will develop an inexact primal-dual stochastic approximation (I-PDSA) method (see Algorithm~2),
which consists of iterative applications of the SPDTs defined in \eqnok{def_dual_extra}, \eqnok{def_primal_proj} and \eqnok{def_dual_proj}
to solve the generic stochastic saddle point problem in \eqnok{primaldual2.2}.

The I-PDSA method evolves from the primal-dual method in \cite{ChamPoc11-1},
an efficient and simple method for solving saddle point problems. While the primal-dual method in  \cite{ChamPoc11-1}
 can be viewed as a refined version of the primal-dual hybrid gradient method
by Arrow et al.~\cite{arrow1958studies}, its design and analysis is more closely related to a few recent important works
which established the ${\cal O}(1/k)$ rate of convergence for solving bilinear saddle point problems (e.g.,~\cite{Nest05-1,Nem05-1,MonSva10-1,he2012on}).
In particular, it is equivalent to a linearized version of the alternative direction method of multipliers.
The first stochastic version of the primal-dual method was studied by Chen, Lan and Ouyang~\cite{CheLanOu13-1} together with
an acceleration scheme and an extension to non-Euclidean projection.
Using a special non-Euclidean geometry, Lan and Zhou~\cite{lan2015optimal} further established an inherent relationship between
the primal-dual method and Nesterov's accelerated gradient method.
However, to the best of our knowledge, none of existing stochastic primal-dual methods can deal with
biased stochastic subgradient information for the value function $\tilde v$.
Moreover, in order to
generate an approximate stochastic subgradient of $V(\cdot, \xi)$ with bounded variance, we will show how to guarantee the boundedness of output dual solution,
while none of existing stochastic optimization methods, including stochastic primal-dual methods,
can guarantee the boundedness of the generated solutions.

\begin{algorithm}  \label{Iter_SPDTs}
\caption{Inexact primal-dual stochastic approximation}
\begin{algorithmic}
\State $\xi = (A,B,b,c)$.
\For {$k = 1,2,\cdots, N$}
\State Let $G_{k-1}$ be a stochastic, independent of $x_{k-1}$, ${\bar \epsilon}$-subgradient of $\tilde v$ , i.e.,
\begin{equation} \label{unbiased_G}
 g(x_{k-1}) \equiv \bbe[G_{k-1}] \in \partial_{\bar \epsilon} \tilde v(x_{k-1}).
\end{equation}
\State $(x_k, y_k, \tilde y_k) = {\rm SPDT}(x_{k-1}, y_{k-1}, y_{k-2}, G_{k-1}, u, \xi, h, X, K_*, \theta_k, \tau_k, \eta_k)$.
\EndFor
\State {\bf Output:} $\bar z_N \equiv (\bar x_N, \bar y_N) = \tsum_{k=1}^N w_k (x_k,y_k) / \tsum_{k=1}^N w_k$.
\end{algorithmic}
\end{algorithm}

Throughout this subsection, we assume that
there exists $M > 0$ such that
\begin{equation}\label{sto_assump1}
\bbe[\|G_k\|_*^2]\leq M^2 \ \  \forall k \ge 1.
\end{equation}
This assumption, in view of \eqnok{unbiased_G} and Jensen's inequality, then implies that
$
\|g(x_{k})\|_*\leq M.
$
For notational convenience, we assume that
the Lipschitz constant of the function $\tilde v$ is also bounded by $M$. Indeed, by definition, any exact subgradient can be viewed as an
${\bar \epsilon}$-subgradient. Hence, the size of subgradient (and the Lipschtiz constant of $\tilde v$) can also be bounded by $M$.
Since the condition in \eqref{def_neighbor_X} about the Lipschitz continuity of the value function $\tilde v$
over a neighborhood of $X$ is hard to verify in practice,
we will discuss different ways to
ensure that the assumption in \eqnok{sto_assump1} holds later in this section (see Corollary~\ref{cor_bnd_dual}).

\vgap

Below we discuss some convergence properties for Algorithm~2. More specifically, we will first establish in Proposition \ref{sto_Prop 1} the relation
between $(x_{k-1},y_{k-1})$ and $(x_k,y_k)$ after running one step of SPDT, and then discuss in Theorems \ref{sto_Theorem 1} and \ref{sto_bounded_thm1} the convergence
properties of Algorithm~2 applied to problem \eqnok{primaldual2.2}.
A few consequences of these results will be discussed in Corollary~\ref{sto_coro 1} and Corollary~\ref{cor_bnd_dual}.
Moreover, we will establish some technical results regarding our termination criterion and the size of the dual multipliers
in Lemma~\ref{lem_cone_prog} and Lemma~\ref{lemma_dual_bnd}, respectively.

\begin{prop}\label{sto_Prop 1}
Let $Q$ be defined in \eqnok{goal_Q}.
For any $1\leq k\leq N$ and $(x,y) \in X \times K_*$, we have
\begin{equation}\label{sto_Q2}
\begin{aligned}
&Q(z_k,z) + \langle A(x_k-x), y_k- y_{k-1} \rangle - \theta_k\langle A(x_{k-1}-x), y_{k-1}- y_{k-2} \rangle\\
 &\leq \tau_k [P_X(x_{k-1},x)-P_X(x_k,x) ]+\tfrac{\eta_k}{2} (\|y-y_{k-1}\|^2-\|y - y_k\|^2) -\tfrac{\alpha_X \tau_k}{2}\|x_k-x_{k-1}\|^2\\
& \quad - \tfrac{\eta_k}{2} \|y_{k-1}-y_k\|^2 +\langle\Delta_{k-1}, x_{k-1}-x\rangle +(M+\|G_{k-1}\|_*)\|x_k-x_{k-1}\| +{\bar \epsilon}\\
& \quad + \theta_k \langle A(x_{k}-x_{k-1}), y_{k-1} - y_{k-2} \rangle,
\end{aligned}
\end{equation}
where
\begin{equation}\label{delta1}
\Delta_k :=g(x_{k})-G_k.
\end{equation}
\end{prop}
\begin{proof}
Denote $\xi = (A,B,b,c)$. By the Lipschitz continuity of $\tilde v$ and the definition of an ${\bar \epsilon}$-subgradient, we have
$$
\begin{aligned}
\tilde v(x_k) 
&\leq \tilde v(x_{k-1})+ M\|x_k - x_{k-1}\|\\
 & \leq \tilde v(x)+\langle g(x_{k-1}),x_{k-1}-x \rangle +  M \|x_k - x_{k-1} \| + {\bar \epsilon}.
\end{aligned}
$$
Moreover, by \eqref{delta1}, we have
\begin{align*}
\langle g(x_{k-1}),x_{k-1}-x \rangle
&= \langle G_{k-1},x_{k-1}-x \rangle  + \langle \Delta_{k-1}, x_{k-1}-x\rangle\\
 &= \langle G_{k-1},x_k-x \rangle +\langle G_{k-1},x_{k-1}-x_k \rangle + \langle \Delta_{k-1}, x_{k-1}-x\rangle\\
 &\leq \langle G_{k-1},x_k-x \rangle+\|G_{k-1}\|_* \|x_k - x_{k-1}\|+ \langle \Delta_{k-1}, x_{k-1}-x\rangle.
\end{align*}
Combining the above two inequalities, 
we obtain
\begin{equation}\label{sto_V_xk}
\begin{aligned}
 \tilde v(x_{k}) - \tilde v(x)
 &\le \langle G_{k-1},x_k-x \rangle + \langle \Delta_{k-1}, x_{k-1}-x\rangle + (M+\|G_{k-1}\|_*)\|x_k - x_{k-1}\| +{\bar \epsilon}.
\end{aligned}
\end{equation}
Moreover, by \eqnok{sto_x_opt0} and \eqnok{sto_y_opt0}
(with input $p = x_{k-1}, d = y_{k-1}, d_{\_} =y_{k-2}, \tilde v' = G_{k-1}, u =u,
h = h , X=X, K_*= K_*, \theta = \theta_k, \tau = \tau_k, \eta =\eta_k$,
output $(p_+, d_+, \tilde d) = (x_k, y_k, \tilde y_k)$, we have
\begin{align}
\langle -A(x_k-x), \tilde y_k \rangle + h(x_k,c)& -h(x,c) + \langle G_{k-1}, x_k-x\rangle \nonumber\\
    &\leq \tau_k [P_X(x_{k-1},x)-P_X(x_k,x) - P_X(x_{k-1},x_k)], \forall x \in X, \label{sto_x_opt}\\
     \langle -b-Bu+Ax_k, y_k-y\rangle &\leq \tfrac{\eta_k}{2} [\|y_{k-1}-y\|^2-\|y_k-y\|^2-\|y_{k-1}-y_k\|^2], \forall y \in K_*.\label{sto_y_opt}
\end{align}
Using the definition of $Q$ in \eqref{goal_Q} and the relations \eqref{sto_V_xk}, \eqref{sto_x_opt} and \eqref{sto_y_opt}, we have
\[
\begin{aligned}
	&Q(z_k,z) +\langle A(x_k-x), y_k-\tilde y_k \rangle \leq\tau_k [P_X(x_{k-1},x)-P_X(x_k,x)] +\tfrac{\eta_k}{2} [\|y_{k-1}-y\|^2-\|y_k-y\|^2] \\
	&- \tau_kP_X(x_{k-1},x_k) -\tfrac{\eta_k}{2}\|y_{k-1}-y_k\|^2+\langle\Delta_{k-1}, x_{k-1}-x\rangle+(M+\|G_{k-1}\|_*)\|x_k-x_{k-1}\|+ {\bar \epsilon}.
\end{aligned}
\]
Also note that by the definition of $\tilde y_k$ (i.e., $\tilde d$ in \eqnok{def_dual_extra}),
we have $\tilde y_k = \theta_k (y_{k-1} - y_{k-2})+ y_{k-1}$ and hence
\[
\begin{aligned}
\langle A(x_k-x), y_k-\tilde y_k \rangle &= \langle A(x_k-x), y_k- y_{k-1} \rangle
- \theta_k \langle A(x_k-x), y_{k-1} - y_{k-2} \rangle\\
&= \langle A(x_k-x), y_k- y_{k-1} \rangle - \theta_k \langle A(x_{k-1}-x), y_{k-1} - y_{k-2} \rangle\\
& \quad- \theta_k \langle A(x_{k}-x_{k-1}), y_{k-1} - y_{k-2} \rangle.
\end{aligned}
\]
Our result then immediately follows from the above two relations and
the strong convexity of $P_X$ (see \eqref{str_con_PD}).
\end{proof}

\vgap

We are now ready to establish some important convergence properties for the iterative applications of SPDTs stated in Algorithm~2.
\begin{thm}\label{sto_Theorem 1}
If the parameters $\{\theta_k\}$, $\{w_k\}$, $\{\tau_k\}$ and $\{\eta_k\}$ in Algorithm~2 satisfy
\begin{equation}\label{sto_cond_1}
\begin{aligned}
w_k\theta_k &= w_{k-1}, 1\leq k\leq N, & (a) \\
 w_k\tau_k &\geq w_{k+1}\tau_{k+1}, 1\leq k\leq N-1, & (b)\\
 w_k\eta_k &\geq w_{k+1}\eta_{k+1}, 1\leq k\leq N-1, & (c) \\
w_k\tau_k\eta_{k-1}\alpha_X&\geq 2w_{k-1}\|A\|^2, 1\leq k\leq N-1, & (d)\\
\tau_N \eta_N \alpha_X &\ge 2 \|A\|^2,&(e)
\end{aligned}
\end{equation}
then we have
\begin{equation}\label{sto_Q_final}
 Q(\bar z_N,z) \leq \tfrac{1}{\tsum_{k=1}^Nw_k} \left(w_1\tau_1 P_X(x_0,x)+ \tfrac{w_1 \eta_1}{2}\|y_0-y\|^2 - \tfrac{w_N\eta_N}{2}\|y_N-y\|^2+ \tsum_{k=1}^N\Lambda_k\right)
\end{equation}
for any $z \in Z$, where
\begin{equation} \label{def_Lambda}
\Lambda_k := w_k\left[(M+\|G_{k-1}\|_*)^2/ (\alpha_X\tau_k)+\langle \Delta_k, x_{k-1}-x\rangle+ {\bar \epsilon}\right].
\end{equation}
\end{thm}

\begin{proof}
Multiplying both sides of \eqref{sto_Q2} by $w_k$ for each $k\geq 1$, summing them up over $1\leq k\leq N$ and using the relations
in \eqnok{sto_cond_1}.a), \eqnok{sto_cond_1}.b) and \eqnok{sto_cond_1}.c), we have
\begin{align}
 &\tsum_{k=1}^N w_kQ(z_k,z) \nn\\
 \leq & w_1\tau_1P_X(x_0,x)+\tfrac{w_1 \eta_1}{2}\|y_0-y\|^2 - \tfrac{w_N\eta_N}{2}\|y_N-y\|^2 +  \tsum_{k=1}^N w_k{\bar \epsilon} \nn \\
  &-w_N\tau_NP_X(x_N,x)-w_N\langle A(x_N-x), y_N- y_{N-1} \rangle -\tfrac{  w_{N}\eta_{N}}{2} \|y_{N}-y_{N-1}\|^2 \nn\\
	& - \tsum_{k=1}^N[\tfrac{\alpha_X w_k \tau_k}{4}\|x_k-x_{k-1}\|^2 +\tfrac{  w_{k-1}\eta_{k-1}}{2} \|y_{k-1}-y_{k-2}\|^2 \nn \\
  &+w_{k-1} \langle A(x_{k}-x_{k-1}), y_{k-1}- y_{k-2} \rangle] - \tsum_{k=1}^N\tfrac{\alpha_X w_k \tau_k}{4}\|x_k-x_{k-1}\|^2 \nn \\
	&+\tsum_{k=1}^N w_k(M+\|G_{k-1}\|_*)\|x_k-x_{k-1}\|+\tsum_{k=1}^Nw_k\langle \Delta_k, x_{k-1}-x\rangle. \label{temp_rel_main}
\end{align}
Now, by the Cauchy-Schwarz inequality and the strong convexity of $P_X$ and \eqnok{sto_cond_1}.e),
\[
\begin{aligned}
& -\tau_NP_X(x_N,x)-\langle A(x_N-x), y_N- y_{N-1} \rangle -\tfrac{  \eta_{N}}{2} \|y_{N}-y_{N-1}\|^2\\
\leq &-\tfrac{\alpha_X \tau_{N}}{2} \|x-x_{N}\|^2+ \|A\| \|x_N-x\| \|y_N- y_{N-1}\| -\tfrac{ \eta_{N}}{2} \|y_{N}-y_{N-1}\|^2 \leq 0.
\end{aligned}
\]
Similarly, by the Cauchy-Schwarz inequality and  \eqnok{sto_cond_1}.d), we have
\[
\begin{aligned}
 & - \tsum_{k=1}^N[\tfrac{\alpha_X w_k \tau_k}{4}\|x_k-x_{k-1}\|^2 +\tfrac{  w_{k-1}\eta_{k-1}}{2} \|y_{k-1}-y_{k-2}\|^2 \\
 & \quad +w_{k-1} \langle A(x_{k}-x_{k-1}), y_{k-1}- y_{k-2} \rangle] \leq 0.
\end{aligned}
\]
Moreover, using the fact that $ -a t^2 /2 + b \le b^2/ (2a)$, we can easily see that
\[
\begin{aligned}
&- \tsum_{k=1}^N\left[\tfrac{\alpha_X \tau_k}{4}\|x_k-x_{k-1}\|^2 + (M+\|G_{k-1}\|_*)\|x_k-x_{k-1}\|\right]
\leq  \tsum_{k=1}^N\tfrac{(M+\|G_{k-1}\|_*)^2}{\tau_k\alpha_X}.
\end{aligned}
\]
Using the above three inequalities in \eqnok{temp_rel_main}, we have
\begin{align*}
 \tsum_{k=1}^N w_kQ(z_k,z)
&\leq w_1\tau_1 P_X(x_0,x)+\tfrac{w_1 \eta_1}{2}\|y_0-y\|^2 - \tfrac{w_N\eta_N}{2}\|y_N-y\|^2  \\
 & \quad + \tsum_{k=1}^N w_k \left(\tfrac{(M+\|G_{k-1}\|_*)^2}{\alpha_X\tau_k} + \langle \Delta_k, x_{k-1}-x\rangle + {\bar \epsilon} \right).
\end{align*}
Dividing both sides of above inequality by $\tsum_{k=1}^N w_k$, and using the convexity of $Q$ and the definition of $\bar z_N$, we obtain \eqref{sto_Q_final}.
\end{proof}

\vgap
\vgap

We also need the following technical result for the analysis of Algorithm~2.
\begin{lem}\label{sto_zkv}
Let $x_0^v \equiv x_0$ and
\begin{equation}\label{sto_def_vv}
x_{k}^v := \argmin_{x\in X}\{\langle \Delta_{k-1},x \rangle + \tau_kP_X(x_{k-1}^v,x)\}
\end{equation}
for any $k \ge 1$.
Then for any $x \in X$,
\begin{equation}\label{sto_lem1}
\begin{aligned}
\tsum_{k=1}^N w_k\langle \Delta_{k-1}, x_{k-1}^v-x \rangle \leq & \tsum_{k=1}^N w_k\tau_k[P_X(x_{k-1},x)-P_X(x_k,x)] +\tsum_{k=1}^N \tfrac{w_k\|\Delta_{k-1}\|_*^2}{2\alpha_X\tau_k}.
\end{aligned}
\end{equation}
\end{lem}

\begin{proof}
It follows from the definition of $x_{k}^v$ in \eqref{sto_def_vv} and Lemma 2.1 of \cite{NJLS09-1} that
$$\tau_kP_X(x_{k}^v,x)\leq \tau_k P_X(x_{k-1}^v,x) - \langle \Delta_{k-1}, x_{k-1}^v-x\rangle + \tfrac{\|\Delta_{k-1}\|_*^2}{2\alpha_X\tau_k},$$
for all $k\geq 1$. Multiplying $w_k$ on both sides of the above inequality and summing them up from $k =1 $ to $N$, we obtain \eqref{sto_lem1}.
\end{proof}

\vgap

Theorem~\ref{sto_bounded_thm1} below provides certain bounds for the following two gap
functions:
\begin{align}
{\rm gap}_*(\bar z) &\equiv {\rm gap}_*(\bar z, X) := \max \left\{ Q(\bar z; x,y_*): x \in X \right\}, \label{def_gap}\\
{\rm gap}_\delta(\bar z)& \equiv {\rm gap}_\delta(\bar z, X, K_*):= \max \left\{Q(\bar z,x,y)+\langle \delta, y\rangle: (x,y)\in X\times K_* \right\}. \label{def_gap_p}
\end{align}
The gap function in \eqnok{def_gap} will be used to measure the error associated with an approximate subgradient, while
the perturbed gap function in \eqnok{def_gap_p} will be used to measure both functional optimality gap and infeasibility
of the conic constraint.
In particular, we will apply the first gap function to the second and third stage, and the latter one to the first stage
when analyzing the DSA algorithm.
\begin{thm}\label{sto_bounded_thm1}
Suppose the parameters $\{\theta_k\}$, $\{w_k\}$, $\{\tau_k\}$ and $\{\eta_k\}$ in Algorithm~2 satisfy \eqref{sto_cond_1}. 
\begin{description}
  \item[a)] For any $N \ge 1$, we have
 \begin{equation} \label{sto_bound1}
\bbe[{\rm gap}_*(\bar z_N)]\leq (\tsum_{k=1}^N w_k)^{-1}\left[2w_1\tau_1\Omega_X^2 + \tfrac{w_1\eta_1}{2}\|y_*-y_0\|^2 + \tsum_{k=1}^N\tfrac{6 w_k M^2}{\alpha_X\tau_k}\right]+ {\bar \epsilon}.
\end{equation}
  \item[b)] If, in addition, $w_1\eta_1 = \ldots = w_N\eta_N$, then
  \begin{align}
 & \bbe[{\rm gap}_\delta(\bar z_N)] \leq (\tsum_{k=1}^N w_k)^{-1}\left[2w_1\tau_1\Omega_X^2 + \tfrac{w_1\eta_1}{2}\|y_0\|^2 + \tsum_{k=1}^N\tfrac{6w_k M^2}{\alpha_X\tau_k}\right] +{\bar \epsilon}, \label{sto_bound2}\\
 & \bbe[\|\delta\|] \leq \tfrac{w_1\eta_1}{\tsum_{k=1}^Nw_k}\left[2\|y_*-y_0\|+2\sqrt{\tfrac{\tau_1}{\eta_1}}\Omega_X
  + \sqrt{\tfrac{2}{w_1\eta_1}\tsum_{k=1}^Nw_k\left(\tfrac{6M^2}{\alpha_X\tau_k}+{\bar \epsilon}\right)}\right], \label{sto_bound3}\\
 & \bbe[\|y_* - \bar y_N\|^2] \le \|y_* - y_0\|^2 + (\tsum_{k=1}^N w_k)^{-1} \tsum_{k=1}^N \tfrac{2}{\eta_k} \left[2 w_1 \tau_1 \Omega_X^2
+ \tsum_{i=1}^k w_i (\tfrac{6M^2}{\tau_i} +{\bar \epsilon}) \right], \label{sto_bound4}
  \end{align}
  where $\delta := (\sum_{k=1}^N w_k)^{-1}[w_1\eta_1(y_0-y_N)].$
\end{description}
\end{thm}

\begin{proof}
We first prove part (a). Letting $y = y_*$ in \eqref{sto_Q_final} and using the definition of $\Omega_X$ in \eqnok{OmegaX}, we have
\begin{equation} \label{spd_temp0}
Q(\bar z_N;x,y_*)\leq (\tsum_{k=1}^N w_k)^{-1} \left[w_1\tau_1\Omega_X^2 + \tfrac{w_1\eta_1}{2}\|y_*-y_0\|^2 - \tfrac{w_N\eta_N}{2}\|y_* - y_N\|^2+ \tsum_{k=1}^N\Lambda_k\right].
\end{equation}
Maximizing w.r.t. $x \in X$ and then taking expectation on both sides of \eqnok{spd_temp1}, we have
\begin{equation} \label{spd_temp1}
\bbe[{\rm gap}_*(\bar z_N)]\leq  (\tsum_{k=1}^N w_k)^{-1} \left[w_1\tau_1\Omega_X^2 + \tfrac{w_1\eta_1}{2}\|y_*-y_0\|^2 + \bbe[\tsum_{k=1}^N\Lambda_k]\right].
\end{equation}
Now it follows from  \eqnok{def_Lambda} and \eqref{sto_lem1} that
\begin{align*}
\tsum_{k=1}^N\Lambda_k &= \tsum_{k=1}^Nw_k\left(\tfrac{(M+\|G_{k-1}\|_*)^2}{\tau_k\alpha_X}+{\bar \epsilon}+\langle \Delta_{k-1}, x_{k-1}-x_{k-1}^v\rangle+\langle \Delta_{k-1}, x_{k-1}^v-x\rangle\right)\\
& \leq \tsum_{k=1}^Nw_k\left(\tfrac{2M^2+2\|G_{k-1}\|_*^2}{\tau_k\alpha_X}+{\bar \epsilon}+\langle \Delta_{k-1}, x_{k-1}-x_{k-1}^v\rangle\right) +w_1\tau_1 \Omega_X^2 +\tsum_{k=1}^N\tfrac{w_k\|\Delta_{k-1}\|_*^2}{2\alpha_X\tau_k}.
\end{align*}
Note that the random noises $\Delta_k$ are independent of $x_{k-1}$ and $\bbe[\Delta_k]=0$, hence $\bbe[\langle \Delta_k,x_{k-1}- x_k^v\rangle] =0$.
Moreover, using the relations that $\bbe[\|G_{k-1}\|_*^2] \le M^2$, $\|g(x_{k-1})\| \le M$ and the triangle inequality, we have
\begin{equation} \label{def_to_be_improved}
\bbe[\|\Delta_{k-1}\|_*^2] = \bbe[\|G_{k-1} - g(x_{k-1})\|_*^2] \le \bbe[(\|G_{k-1}\|_* + \|g(x_{k-1})\|_*)^2] \le 4M^2.
\end{equation}
Therefore,
\begin{equation} \label{bound_sum_lambda}
\bbe[\tsum_{k=1}^N\Lambda_k] \le w_1\tau_1 \Omega_X^2 +\tsum_{k=1}^N w_k \left(\tfrac{6 M^2 }{\alpha_X\tau_k} +  {\bar \epsilon} \right).
\end{equation}
The result \eqnok{sto_bound1} then follows by using the above relation in \eqnok{spd_temp1}.

We now show part (b) holds. Adding $\langle \delta, y \rangle$
to both sides of \eqref{sto_Q_final} and using the fact that $w_1\eta_1 = w_N \eta_N$,
we have
\[
\begin{aligned}
Q(\bar z_N, z) + \langle \delta, y \rangle
&\le (\tsum_{k=1}^N w_k)^{-1} [w_1\tau_1 P_X(x_0,x)+ w_1 \eta_1\left(\tfrac{1}{2}\|y_0-y\|^2 - \tfrac{1}{2}\|y_N-y\|^2+
\langle y_0-y_N, y \rangle\right) \\
& \quad \quad \quad +  \tsum_{k=1}^N\Lambda_k]\\
&\le (\tsum_{k=1}^N w_k)^{-1} [w_1\tau_1 P_X(x_0,x)+ \tfrac{w_1\eta_1}{2}\|y_0\|^2 +  \tsum_{k=1}^N\Lambda_k].
\end{aligned}
\]
Maximizing both sides of the above inequality w.r.t. $(x,y) \in X \times K_*$, taking expectation
and using \eqnok{def_gap_p},
we obtain
\[
 \bbe[{\rm gap}_\delta(\bar z_N)] \le(\tsum_{k=1}^N w_k)^{-1} \left [w_1\tau_1 \Omega_X^2+ \tfrac{w_1\eta_1}{2}\|y_0\|^2 +  \bbe[ \tsum_{k=1}^N\Lambda_k]\right].
\]
The result in \eqnok{sto_bound2} then follows from the above inequality and \eqnok{bound_sum_lambda}.
Now fixing $x = x_*$ in \eqnok{spd_temp0} and using the fact $Q(\bar z_N;x_*,y_*)\geq 0$, we have
\[
 \tfrac{w_N\eta_N}{2}\|y_* - y_N\|^2\leq w_1\tau_1\Omega_X^2 + \tfrac{w_1\eta_1}{2}\|y_*-y_0\|^2 + \tsum_{k=1}^N\Lambda_k.
\]
Taking expectation on both sides of the above inequality and using \eqnok{bound_sum_lambda}, we conclude
\begin{equation} \label{bound_dual_iter}
\tfrac{w_N \eta_N}{2}\bbe[\|y_*-y_N\|^2]\leq 2w_1\tau_1\Omega_X^2+\tfrac{w_1\eta_1}{2}\|y_*-y_0\|^2+ \tsum_{k=1}^Nw_k\left(\tfrac{6M^2}{\alpha_X\tau_k}+{\bar \epsilon}\right),
\end{equation}
which implies that
\[
\bbe[\|y_*-y_N\|]\leq 2\sqrt{\tfrac{\tau_1}{\eta_1}}\Omega_X+\|y_*-y_0\|+ \sqrt{\tfrac{2}{w_1\eta_1}\tsum_{k=1}^Nw_k\left(\tfrac{6M^2}{\alpha_X\tau_k}+{\bar \epsilon}\right)}.
\]
Using the above inequality and the fact that
$\|\delta\| \leq (\sum_{k=1}^N w_k)^{-1}[w_1\eta_1(\|y_0- y_*\| + \|y_*-y_N\|)$, we obtain \eqnok{sto_bound3}.
Observe that \eqnok{bound_dual_iter} holds for any $y_k$, $k = 1, \ldots, N$, and hence that
\[
\tfrac{w_k \eta_k}{2}\bbe[\|y_*-y_k\|^2]\leq 2w_1\tau_1\Omega_X^2+\tfrac{w_1\eta_1}{2}\|y_*-y_0\|^2+ \tsum_{i=1}^k w_i\left(\tfrac{6M^2}{\alpha_X\tau_i}+{\bar \epsilon}\right).
\]
Using the above inequality, the convexity of $\|\cdot\|^2$ and the fact that $\bar y_N = \tsum_{k=1}^N (w_ky_k) / \tsum_{k=1}^N w_k$, we conclude that
\begin{align*}
\bbe[\|y_* - \bar y_N\|^2] &\le (\tsum_{k=1}^N w_k)^{-1} \tsum_{k=1}^N \left[\tfrac{4 w_1 \tau_1 \Omega_X^2}{\eta_k} + \tfrac{w_1 \eta_1}{\eta_k} \|y_* - y_0\|^2
+\tfrac{2}{\eta_k} \tsum_{i=1}^k w_i (\tfrac{6M^2}{\tau_i} +{\bar \epsilon}) \right]\\
&= \|y_* - y_0\|^2 + (\tsum_{k=1}^N w_k)^{-1} \tsum_{k=1}^N \left[\tfrac{4 w_1 \tau_1 \Omega_X^2}{\eta_k}
+\tfrac{2}{\eta_k} \tsum_{i=1}^k w_i (\tfrac{6M^2}{\tau_i} +{\bar \epsilon}) \right],
\end{align*}
where the second identity follows from the fact that $w_k \eta_k = w_1 \eta_1$.
\end{proof}

\vgap

Below we provide two different parameter settings for $\{w_k\},\{\tau_k\}$ and $\{\eta_k\}$ satisfying \eqref{sto_cond_1}.
While the first one in Corollary~\ref{sto_coro 1} leads to slightly better rate of convergence, the second one in
Corollary~\ref{cor_bnd_dual} can guarantee the boundedness of the dual solution in expectation.
We will discuss how to use these results when analyzing the convergence of the DSA algorithm.

\begin{cor}\label{sto_coro 1}
If
\begin{equation}\label{step1}
w_k = w = 1, \tau_k =\tau = \max\{\tfrac{M \sqrt{3 N}}{\Omega_X\sqrt{\alpha_X}},\tfrac{\sqrt{2}\|A\|}{\sqrt{\alpha_X}}\} \text{ and } \eta_k =\eta= \tfrac{\sqrt{2}\|A\|}{\sqrt{\alpha_X}}, \forall 1\leq k \leq N,
\end{equation}
then
  \begin{align}
  \bbe[{\rm gap}_*(\bar z_N)] &\leq  \tfrac{\sqrt{2}\|A\| (2\Omega_X^2+\|y_*-y_0\|^2)}{\sqrt{\alpha_X} N}+\tfrac{4\sqrt{3} M\Omega_X}{\sqrt{\alpha_X N}}+{\bar \epsilon},\label{g_bound0}\\
  \bbe[{\rm gap}_\delta(\bar z_N)] &\leq  \tfrac{\sqrt{2}\|A\| (2\Omega_X^2+\|y_0\|^2)}{\sqrt{\alpha_X} N}+\tfrac{4\sqrt{3} M\Omega_X}{\sqrt{\alpha_X N}} +{\bar \epsilon},\label{g_bound}\\
  \bbe[\|\delta\|] &\leq \tfrac{2\sqrt{2\alpha_X}\|A\| \|y_* - y_0\| + 4 \Omega_X \|A\|}{\alpha_X N}
+ \tfrac{2M (\sqrt{6} \|A\| + \sqrt{3\alpha_X}) }{\alpha_X \sqrt{N}} + \sqrt{\tfrac{3 \|A\| {\bar \epsilon}}{N \sqrt{\alpha_X}}}, \label{g_bound1}\\
 \bbe[\|y_* - \bar y_N\|^2] &\le
  \|y_* - y_0\|^2 + 4 \Omega_X^2 + \tfrac{2 \sqrt{6N} M \Omega_X}{\|A\|} + \tfrac{3 \alpha_X (N+1) M^2}{\|A\|^2} + \tfrac{(N+1) {\bar \epsilon}}{2}.\label{g_bound2}
\end{align}
\end{cor}
\begin{proof}
We can easily check that the parameter setting in \eqnok{step1} satisfies \eqref{sto_cond_1}.
It follows from \eqnok{sto_bound1} and \eqnok{step1} that
\begin{align*}
\bbe[{\rm gap}_*(\bar z_N)] &\le \tfrac{1}{N} \left[2\tau\Omega_X^2 + \tfrac{\eta}{2}\|y_*-y_0\|^2 +\tfrac{6N M^2 }{ \alpha_X \tau}\right] + {\bar \epsilon}
 \leq \tfrac{\sqrt{2}\|A\| (2\Omega_X^2+\|y_*-y_0\|^2)}{\sqrt{\alpha_X} N}+\tfrac{4\sqrt{3} M\Omega_X}{\sqrt{\alpha_X N}} +{\bar \epsilon}.
\end{align*}
Moreover, we have  $w_1\eta_1 = w_N\eta_N$. Hence, by
\eqnok{sto_bound2} and \eqnok{step1},
\begin{align*}
\bbe[{\rm gap}_\delta(\bar z_N)] & \le \tfrac{1}{N} \left[2\tau\Omega_X^2 + \tfrac{\eta}{2}\|y_0\|^2 +\tfrac{6N M^2 }{ \alpha_X \tau}\right] + {\bar \epsilon}
 \le  \tfrac{\sqrt{2}\|A\| (2\Omega_X^2+\|y_0\|^2)}{\sqrt{\alpha_X} N}+\tfrac{4\sqrt{3} M\Omega_X}{\sqrt{\alpha_X N}} +{\bar \epsilon}.
\end{align*}
Also by \eqnok{sto_bound3} and \eqnok{step1},
\begin{align*}
\bbe[\|\delta\|] &\leq \tfrac{\eta}{N}\left[2\|y_*-y_0\|+2\sqrt{\tfrac{\tau}{\eta}}\Omega_X
  + \sqrt{\tfrac{2 N}{\eta}\left(\tfrac{6M^2}{\alpha_X\tau}+{\bar \epsilon}\right)}\right]\\
  &\le \tfrac{2 \sqrt{2} \|A\| \|y_*- y_0\|}{N \sqrt{\alpha_X}} + \tfrac{2 \Omega_X}{N}\left(\tfrac{2 \|A\|}{\alpha_X} + \tfrac{\sqrt{6N} \|A\| M}{\Omega_X \alpha_X} \right)
  + \tfrac{2 \sqrt{M}}{\sqrt{\alpha_X N}} + \tfrac{\sqrt{2 {\bar \epsilon}}}{\sqrt{N}} \sqrt{\tfrac{\sqrt{2}\|A\|}{\sqrt{\alpha_X}}},
\end{align*}
which implies \eqnok{g_bound1}.
Finally, by \eqnok{sto_bound3} and \eqnok{step1},
\begin{align*}
  \bbe[\|y_* - \bar y_N\|^2] &\le \|y_* - y_0\|^2 + \tfrac{1}{N} \left[\tsum_{k=1}^N \tfrac{4 \tau_k}{\eta_k}\Omega_X^2
  + \tsum_{k=1}^N \tfrac{2}{\eta_k} \tsum_{i=1}^k \left( \tfrac{6M^2}{\tau_i} +{\bar \epsilon} \right)\right]\\
  &\le \|y_* - y_0\|^2  + 4 \Omega_X^2 + \tfrac{2 \sqrt{6N} M \Omega_X}{\|A\|} + \tfrac{3 \alpha_X (N+1) M^2}{\|A\|^2} + \tfrac{(N+1) {\bar \epsilon}}{2}.
\end{align*}
%
%
\end{proof}

\vgap

In view of \eqnok{g_bound2}, if $M >0$ or $N$ is not properly chosen, we cannot guarantee that $\bbe[\|y_* - \bar y_N\|^2]$ is bounded.
In the following corollary, we will modify the selection of $\tau$ and $\eta$ in \eqnok{step1}
in order to guarantee the boundedness of $\bbe[\|y_* - \bar y_N\|^2]$ even when $M > 0$.
\begin{cor} \label{cor_bnd_dual}
If
\begin{equation}\label{step2}
w_k = w = 1, \tau_k =\tau = \max\{\tfrac{M \sqrt{3 N}}{\Omega_X\sqrt{\alpha_X}},\tfrac{\sqrt{2}\|A\| }{\sqrt{\alpha_X N}}\} \text{ and } \eta_k =\eta= \tfrac{\sqrt{2 N}\|A\|}{\sqrt{\alpha_X}}, \forall 1\leq k \leq N,
\end{equation}
then
  \begin{align}
  \bbe[{\rm gap}_*(\bar z_N)] &\leq \tfrac{2 \sqrt{2} \|A\| \Omega_X^2}{N \sqrt{\alpha_X N}}
     + \tfrac{\|A\| \|y_* - y_0\|^2 + 4 \sqrt{3} M \Omega_X}{\sqrt{\alpha_X N}} +{\bar \epsilon},\label{sto_cor2_a}\\
  \bbe[{\rm gap}_\delta(\bar z_N)] &\leq  \tfrac{2 \sqrt{2} \|A\| \Omega_X^2}{N \sqrt{\alpha_X N}}
     + \tfrac{\|A\| \|y_0\|^2 + 4 \sqrt{3} M \Omega_X}{\sqrt{\alpha_X N}} +{\bar \epsilon}, \label{sto_cor2_b}\\
  \bbe[\|\delta\|] &\leq \tfrac{2 \sqrt{2} \|A\| \|y_*- y_0\| + 4\sqrt{ M \|A\| \Omega_X}}{\sqrt{\alpha_X N}} + \tfrac{2\sqrt{6} \|A\| M}{\alpha_X} + \tfrac{4 \Omega_X^2 \|A\|^2}{N \alpha_X}
+ \sqrt{\tfrac{3 \|A\|{\bar \epsilon}}{\sqrt{\alpha_X N} }},\label{sto_cor2_c}\\
\bbe[\|y_* - \bar y_N\|^2] &\le    \|y_* - y_0\|^2 + \tfrac{2 \Omega_X^2}{N} + \tfrac{\sqrt{6} (1+\alpha_X) M \Omega_X}{\|A\|}  + \tfrac{\sqrt{\alpha_X N} {\bar \epsilon}}{\sqrt{2} \|A\|}. \label{sto_cor2_d}
  \end{align}
\end{cor}

\begin{proof}
The proofs of \eqref{sto_cor2_a}-\eqref{sto_cor2_d} are similar to Corollary \ref{sto_coro 1} and hence the details are skipped.
\end{proof}

\vgap

Note that by using the parameter setting \eqref{step2}, we still obtain the optimal rate of convergence in terms of the dependence on $N$,
with a slightly worse dependence on $\|A\|$ and $\|y_*\|$ than the one obtained by using the parameter setting in \eqnok{step1}. However,
using the setting \eqref{step2}, we can bound $\bbe[\|\bar y_N - y_*\|^2]$ as long as $N = {\cal O} (1/{\bar \epsilon}^2)$, while this statement does not necessarily hold
for the parameter setting in \eqnok{step1}.

\vgap

We now state one technical result regarding the functional optimality gap and primal infeasibility, which generalizes
Proposition 2.1 of \cite{OuCheLanPas14-1} to conic programming.
\begin{lem} \label{lem_cone_prog}
If there exist random vectors $\delta \in \bbr^m$ and $\bar z \equiv (\bar x, \bar y) \in Z$ such that
\begin{equation} \label{bnd_primal_cond}
\bbe[{\rm gap}_\delta(\bar z)] \le \epsilon_o,
\end{equation}
then
\[
  \begin{array}{l}
  \bbe[h(\bar x, c) + \tilde v(\bar x) - (h(x^*, c) + \tilde v(x^*)) ] \le \epsilon_0,\\
   A \bar x - Bu - b - \delta \in K \ {\rm a.s.},
\end{array}
\]
where $x^*$ is an optimal solution of  problem~\eqnok{primaldual2.2}.
\end{lem}

\begin{proof}
Letting $x = x^*$ and $y = 0$ in the definition of \eqnok{def_gap_p}, we can easily see that
\[
h(\bar x, c) + \tilde v(\bar x) - (h(x^*, c) + \tilde v(x^*)) \le {\rm gap}_\delta(\bar z).
\]
Moreover, in view of \eqnok{goal_Q} and \eqnok{def_gap_p}, we must have
$A \bar x  - Bu  -b - \delta \in K$ almost surely. Otherwise, $\bbe[{\rm gap}_\delta(\bar z)]$
would be unbounded as $y$ runs throughout $K^*$ in the definition of ${\rm gap}_\delta(\bar z)$.
\end{proof}

\vgap

In the next result, we will provide a bound on the optimal dual variable $y_*$. By doing so, we show that the complexity
of Algorithm~2 only depends on the parameters for the primal problem along with the smallest nonzero eigenvalue of $A$ and the initial point $y_0$,
even though the algorithm is a primal-dual type method.

\begin{lem} \label{lemma_dual_bnd}
Let $(x^*, y^*)$ be an optimal solution to problem \eqref{primaldual2.2}.
If the subgradients of the objective function $v_h(x) := h(x, c) + \tilde v(\cdot)$ are bounded, i.e.,
$\|v_h'(x)\|_2 \le M_h$ for any $x \in X$,
then there exists $y^*$ s.t.
\begin{equation}\label{boundedness_y}
\|y^*\|\leq \tfrac{M_h}{\sigma_{min}(A)},
\end{equation}
where $\sigma_{min}(A)$ denotes the smallest nonzero singular value  of $A$.
\end{lem}

\begin{proof}
We consider two cases. Case 1: $A^Ty^*=0$, i.e., $y_*$ belongs to the null space of $A$. Since for any $\lambda\geq 0$, $\lambda y^*$ is still
an optimal dual solution to problem \eqref{primaldual2.2}, we have \eqref{boundedness_y} holds.\newline
Case 2: $A^T y^* \neq 0$.
By the definition of the saddle point, we have
\[
\langle b + B u - A x^*,  y^*\rangle +  h(x^*,c) + \tilde v(x^*) \le \langle b + B u - A x,  y^*\rangle +  h(x,c) + \tilde v(x), \ \forall x \in X,
\]
which implies
\begin{equation} \label{temp_subgradeint_vh}
h(x^*,c) + \tilde v(x^*) + \langle A^T y^*, x - x^* \rangle \le h(x,c) + \tilde v(x), \ \forall x \in X.
\end{equation}
Hence $A^T y^*$ is a subgradient of $v_h$ at the point $x^*$.
Without loss of generality, we assume that $y^*$ belongs to the column space of $A^T$ (i.e., $y^*$ is perpendicular
to the eigenspace associated with eigenvalue $0$). Otherwise we can show that the projection of $y^*$
onto the column space of $A^T$ will also satisfy \eqnok{temp_subgradeint_vh}. Using this observation, we have
\[
\|A^T y^*\|_2^2 =(y^*)^TAA^Ty^* = (y^*)^TU^T\Lambda Uy^*\geq \sigma_{min}(AA^T)\|Uy^*\|^2 = \sigma_{min}^2(A)\|y^*\|^2,
\]
where $U$ is an orthonormal matrix whose rows consist of the eigenvectors of $A A^T$ and
$\Lambda$ is the diagonal matrix whose elements are the corresponding eigenvalues.
Our result then follows from the above inequality and the assumption that $\|A^T y^*\|_2 \le M_h$.
\end{proof}

\subsection{Convergence analysis for DSA} \label{sec_con_DSA}
Our goal in this subsection is to establish the complexity of the DSA algorithm for solving problem~\ref{first-stage-cp}.

The basic idea is to apply the results we obtained in the previous section regarding the I-PDSA algorithm to
the three loops stated in the DSA algorithm. More specifically, we will show how to generate
stochastic $\epsilon$-subgradients for the value functions $v^2$ and $v^3$
in the middle and innermost loops, respectively, and how to compute a nearly optimal solution
for problem~\ref{first-stage-cp} in the outer loop of the DSA algorithm .

In order to apply these results to the saddle-point reformulation for
the second and first stage problems (see \eqnok{primaldual2} and \eqnok{primaldual1}),
we need to make sure that the condition in \eqnok{sto_assump1}
holds for the value functions, $v^3$ and $v^2$ respectively, associated with the optimization problems in their subsequent stages.
For this purpose, we assume that the less aggressive algorithmic parameter setting in \eqnok{step2}
is applied to solve the second stage saddle point problems in \eqnok{primaldual2},
while a more aggressive
parameter setting in \eqnok{step1} is used to solve the first stage and last stage saddle point problems in \eqnok{primaldual1}
and \eqnok{primaldual3}, respectively. Moreover, we need the boundedness of the operators $B^2$ and $B^3$:
\begin{equation} \label{boundedB3}
\|B^2\| \le {\cal B}_2  \ \ \mbox{and} \ \ \ \|B^3\| \le {\cal B}_3
\end{equation}
in order to guarantee that the generated stochastic subgradients for the value functions $v^2$ and $v^3$ have bounded variance.

For notational convenience, we use $\Omega_i \equiv \Omega_{X^i}$ and $\alpha_i \equiv \alpha_{X^i}$, $i=1, 2, 3$, to denote the
diameter and strongly convex modulus associated with the distance generating function
for the feasible set $X^i$ (see \eqnok{OmegaX}).
Lemma~\ref{lem_analysis_stage3} shows some convergence properties for the innermost loop of the DSA algorithm.
\begin{lem} \label{lem_analysis_stage3}
If the parameters $\{w_k^3\}$, $\{\tau_k^3\}$ and $\{\eta_k^3\}$ are set to \eqref{step1} (with $M = 0$ and $A = A_j^3$) and
\begin{equation} \label{N3}
N_3 \equiv N_{3,j} :=  \tfrac{3\sqrt{2}\|A_j^3\| [2(\Omega_3)^2+\|y^3_{*,j}-y^3_{0}\|^2]}{\sqrt{\alpha_3} \epsilon},
\end{equation}
then $B_j^3 \bar y_j^3$ is a stochastic $(\epsilon/3)$-subgradient of the value function $v^3$ at $x_{j-1}^2$.
Moreover, given random variable $\xi^{[2]}$, there exists a constant $M_3$
such that $\|v^3(x_1, \xi^{[2]}) - v^3 (x_2,\xi^{[2]}) \| \le M_3 \|x_1 - x_2\|, \forall x_1, x_2 \in X^2$ and
\begin{equation} \label{def_M3}
\bbe[\|B_j^3 \bar y_j^3\|_*^2|\xi^{[2]}] \le M_3^2.
\end{equation}
In addition, there exists a vector $\delta \in \bbr^{m^3}$ s.t.
\begin{equation} \label{bnd_gap_stage3}
  \begin{array}{l}
 \bbe[h^3(\bar x^3, c^3) - V^3(\bar x^2, \xi^{[3]})  | \xi^{[2]}] \le \epsilon/3,\\
   A^3 \bar x^3 - B^3\bar x^2 - b^3 - \delta \in K^3 \ {\rm a.s.},\\
   \bbe[\|\delta\| | \xi^{[2]}] \le \epsilon/3.
\end{array}
\end{equation}
\end{lem}

\begin{proof}
The innermost loop of the DSA algorithm is equivalent to the application of Algorithm~2 to
the last stage saddle point problem in~\eqnok{primaldual3}.
Note that for this problem, we do not have any subsequent stages and hence $\tilde v = 0$. In other words,
the subgradients of $\tilde v$ are exact. In view of Corollary~\ref{sto_coro 1} (with $M = 0$ and $\bar \epsilon = 0$),
the definition of $N_3$
in \eqnok{N3} and conditional on $\xi^{[2]}$, we have
\[
\bbe[{\rm gap}_*(\bar z_j^3)|\xi^{[2]}]  \le \tfrac{\sqrt{2}\|A_j^3\| [2(\Omega_3)^2+\|y^3_*-y^3_0\|^2]}{\sqrt{\alpha_3} N_3} \le \tfrac{\epsilon}{3}.
\]
This observation, in view of Lemma~\ref{epsi_subg}, then implies that
$B_j^3 \bar y_j^3$ is a stochastic $(\epsilon/3)$-subgradient of $v^3$ at $x_{j-1}^2$.
By the Lipschitz continuity of $v^3$, the Lipschitz constant $M_3$ should satisfy
\begin{equation}\label{M3_1}
M_3 \ge \bbe[\|B_j^3y_{*,j}^3\||\xi^{[2]}],\ \forall y_{*,j}^3\in Y_*^3,
\end{equation}
where $Y_*^3$ denotes the set of optimal dual solutions of problem \eqref{primaldual3}.
Moreover, it follows from \eqnok{g_bound2} (with $M = 0$ and ${\bar \epsilon}=0$) that
\[
\begin{array}{l}
\bbe[\|y_{*,j}^3 - \bar y_{j}^3\|^2|\xi^{[2]}]  \le \bbe[\|y_{*,j}^3 - y_0^3\|^2 |\xi^{[2]}] + 4 (\Omega_3)^2, \\
\bbe [\|\bar y_{j}^3\|^2|\xi^{[2]}]  \le 2 \bbe[\|y_{*,j}^3 \| +\|y_{*,j}^3 - y_0^3\|^2|\xi^{[2]}]+8\Omega_3^2.
\end{array}
\]
This inequality, in view of \eqref{boundedB3}, implies that
\begin{equation}\label{M3_2}
\bbe[\|B_j^3 \bar y_j^3\|_*^2|\xi^{[2]}] \le {\cal B}_3^2 \bbe [(2\|y_{*,j}^3 \| +2\|y_{*,j}^3 - y_0^3\|^2 +8\Omega_3^2) |\xi^{[2]}].
\end{equation}
Hence, combining \eqref{def_M3}, \eqref{M3_1} and \eqref{M3_2}, we can see that the latter part of our result holds with 
\[
M_3 = \max \left\{ \max_{y\in Y_*^3} \bbe[\|B_j^3y\||\xi^{[2]}], {\cal B}_3 \sqrt{\bbe[(2\|y_{*,j}^3 \| + 2\|y_{*,j}^3 - y_0^3\|^2+8\Omega_3^2) |\xi^{[2]}] } \right\}.
\]
The results in \eqnok{bnd_gap_stage3} directly follow from Lemma~\ref{lem_cone_prog}.
In view of Corollary~\ref{sto_coro 1} (with $M = 0$ and $\bar \epsilon = 0$)
and the definition of $N_3$
in \eqnok{N3}, we conclude that there exist $\delta \in  \bbr^{m^1}$ s.t.
\begin{align*}
  \bbe_{\xi^2}[\|\delta\|] &\leq \tfrac{2\sqrt{2\alpha_3}\|A^3\| \|y_*^3 - y_0^3\| + 4 \Omega_3 \|A^3\|}{\alpha_3 N_3}
 \le \epsilon/3,
\end{align*}
which together with Lemma~\ref{lem_cone_prog} then imply our result.
\end{proof}

\vgap

Lemma~\ref{lem_analysis_stage2} describes some convergence properties for the middle loop of the DSA algorithm.
\begin{lem} \label{lem_analysis_stage2}
Assume that the parameters for the innermost loop are set according to Lemma~\ref{lem_analysis_stage3}.
If the parameters $\{w_j^2\}$, $\{\tau_j^2\}$ and $\{\eta_j^2\}$ for the middle loop are set to \eqref{step2} (with $M = M_3$ and $A = A_i^2$) and
\begin{equation} \label{N2}
N_2 \equiv N_{2,i} :=  \left(\tfrac{12 \sqrt{2} \|A^2_i\| \Omega_2}{\sqrt{\alpha_2} \epsilon} \right)^\frac{2}{3}
     + \left[\tfrac{6\left(\|A^2_i\| \|y_{*,i}^2 - y_0^2\|^2 + 4 \sqrt{3} M_3 \Omega_2\right)}{\sqrt{\alpha_2} \epsilon}\right]^2,
\end{equation}
then $B_i^2 \bar y_i^2$ is a stochastic $(2\epsilon/3)$-subgradient of the value function $v^2$ at $x_{i-1}^1$.
Moreover, there exists a constant $M_2$
such that $\|v^2(x_1) - v^2 (x_2) \| \le M_2 \|x_1 - x_2\|, \forall x_1, x_2 \in X^2$ and
\begin{equation} \label{def_M2}
\bbe[\|B_i^2 \bar y_i^2\|_*^2|\xi^{[1]}] \le M_2^2,
\end{equation}
In addition, there exists a vector $\delta \in \bbr^{m^2}$ s.t.
\[
  \begin{array}{l}
  \bbe[h^2(\bar x^2, c^2) + v^3(\bar x^2|\xi^2) - V^2(\bar x^1, \xi^{[2]}) ] \le 2\epsilon/3,\\
   A^2 \bar x^2 - B^2\bar x^1 - b^2 - \delta \in K^2 \ {\rm a.s.}, \\
    \bbe[\|\delta\| | \xi^{[2]}] \le 2\epsilon/3.
\end{array}
\]
\end{lem}

\begin{proof}
The middle loop of the DSA algorithm is equivalent to the application of Algorithm~2 to
the second stage saddle point problem in~\eqnok{primaldual2}.
Note that for this problem, we have $\tilde v = v^3$. Moreover, by Lemma~\ref{lem_analysis_stage3},
the stochastic subgradients of $v^3$ are computed by the
innermost loop with tolerance $\bar \epsilon = \epsilon/3$. In view of Corollary~\ref{cor_bnd_dual} (with $M = M_3$ and $\bar \epsilon = \epsilon/3$)
and the definition of $N_2$
in \eqnok{N2}, we have
\[
\bbe[{\rm gap}_*(\bar z_i^2)|\xi^{[1]}]  \le \tfrac{2 \sqrt{2} \|A^2_i\| \Omega_2}{N_2 \sqrt{\alpha_2 N_2}}
     + \tfrac{\|A^2_i\| \|y_{*,i}^2 - y_0^2\|^2 + 4 \sqrt{3} M_3 \Omega_2}{\sqrt{\alpha_2 N_2}} +{\bar \epsilon}\le \tfrac{2\epsilon}{3}.
\]
This observation, in view of Lemma~\ref{epsi_subg}, then implies that
$B_i^2 \bar y_i^2$ is a stochastic $(2\epsilon/3)$-subgradient $v^2$ at $x_{i-1}^1$.
By the Lipschitz continuity of $v^2$, the Lipschitz constant $M_2$ should satisfy
\begin{equation}\label{M2_1}
M_2 \ge \bbe[\|B_i^2y_{*,i}^2\| | \xi^{[1]}],\ \forall y_{*,i}^2\in Y_*^2,
\end{equation}
where $Y_*^2$ denotes the set of optimal dual solutions of problem \eqref{primaldual2}.
Moreover, it follows from \eqnok{sto_cor2_d} (with $M = M_3$ and ${\bar \epsilon}=\epsilon/3$) that
\[
\begin{array}{l}
\bbe[\|y_{*,i}^2 - \bar y_i^2\|^2| \xi^{[1]} ]\le \bbe[ \|y_{*,i}^2 - y_0^2\|^2 + \tfrac{2 \Omega_2^2}{N_2} + \tfrac{\sqrt{6} (1+\alpha_2) M_3 \Omega_2}{\|A^2_i\|}
+ \tfrac{\sqrt{\alpha_2 N_2} { \epsilon}}{3\sqrt{2} \|A^2_i\|}|\xi^{[1]}]],\\
\bbe[\| \bar y_i^2\|^2| \xi^{[1]}]\le \bbe[2\|y_{*,i}^2\|^2 + 2\|y_{*,i}^2 - y_0^2\|^2 + \tfrac{4 \Omega_2^2}{N_2} + \tfrac{2\sqrt{6} (1+\alpha_2) M_3 \Omega_2}{\|A^2_i\|}+ \tfrac{\sqrt{2\alpha_2 N_2} { \epsilon}}{3 \|A^2_i\|}| \xi^{[1]}].
\end{array}
\]
This inequality, in view of \eqref{boundedB3}, implies that
\begin{equation}\label{M2_2}
\bbe[\|B_i^2 \bar y_i^2\|_*^2| \xi^{[1]}] \le {\cal B}_2^2 \bbe\left[2\|y_{*,i}^2\|^2 + 2\|y_{*,i}^2 - y_0^2\|^2 + \tfrac{4 \Omega_2^2}{N_2} + \tfrac{2\sqrt{6} (1+\alpha_2) M_3 \Omega_2}{\|A^2_i\|}+ \tfrac{\sqrt{2\alpha_2 N_2} { \epsilon}}{3 \|A^2_i\|}| \xi^{[1]}\right],
\end{equation}
where $N_2$ is defined in \eqnok{N2}.
Hence, combining these observations, we can see that the latter part of our results holds with
$M_2$ satisfying both \eqnok{M2_1} and
\[
M_2 \ge {\cal B}_2 \left\{\bbe\left[2\|y_{*,i}^2\|^2 + 2\|y_{*,i}^2 - y_0^2\|^2 + \tfrac{4 \Omega_2^2}{N_2} + \tfrac{2\sqrt{6} (1+\alpha_2) M_3 \Omega_2}{\|A^2_i\|}+ \tfrac{\sqrt{2\alpha_2 N_2} { \epsilon}}{3 \|A^2_i\|}| \xi^{[1]}\right]\right\}^\frac{1}{2}.
\]
In view of Corollary~\ref{sto_coro 1} (with $M = M_3$ and $\bar \epsilon = \epsilon/3$)
and the definition of $N_2$
in \eqnok{N2}, we conclude that there exist $\delta \in  \bbr^{m^1}$ s.t.
\begin{align*}
  \bbe_{\xi^2}[\|\delta\|] &\leq \tfrac{2\sqrt{2\alpha_2}\|A^2\| \|y_*^2 - y_0^2\| + 4 \Omega_2 \|A^2\|}{\alpha_2 N_2}
+ \tfrac{2M_2 (\sqrt{6} \|A^2\| + \sqrt{3\alpha_2}) }{\alpha_2 \sqrt{N_2}} + \sqrt{\tfrac{2 \|A^2\|  \epsilon}{N_2 \sqrt{\alpha_2}}} \le 2\epsilon/3,
\end{align*}
which together with Lemma~\ref{lem_cone_prog} then imply our result.
\end{proof}

%
\vgap

We are now ready to establish the main convergence properties of the DSA algorithm applied to a three-stage
stochastic optimization problem.
\begin{thm}\label{sto_bound_all}
Suppose that the parameters for the innermost and middle loop in the DSA algorithm are set according to Lemma~\ref{lem_analysis_stage3}
and Lemma~\ref{lem_analysis_stage2}, respectively. If
the parameters $\{w_i\}$, $\{\tau_i\}$ and $\{\eta_i\}$ for the outer loop are set to \eqnok{step1} (with $M = M_2$ and $A = A^1$) and
\begin{equation} \label{N1}
\begin{array}{l}
N_1 := \max \left\{ \tfrac{6\sqrt{2}\|A^1\| [2(\Omega_1)^2+\|y_0^1\|^2]}{\sqrt{\alpha_1} \epsilon}+ \left(\tfrac{24\sqrt{3} M_2\Omega_1}{\sqrt{\alpha_1} \epsilon}\right)^2, \right.\\
\quad \quad \left. \tfrac{6\|A^1\| (\sqrt{2\alpha_1} \|y_*^1 - y_0^1\| + 2 \Omega_1 + 3 \sqrt{\alpha_1})}{\alpha_1 \epsilon} + \left( \tfrac{6 \sqrt{3} M_2 (\sqrt{2} \|A^1\| + \sqrt{\alpha_1}) }{\alpha_1 \epsilon}\right)^2
\right\},
\end{array}
\end{equation}
then we will find a solution $\bar x^1 \in X^1$ and a vector $\delta \in \bbr^{m^1}$ s.t.
\[
  \begin{array}{l}
  \bbe[h(\bar x^1, c) + v^2(\bar x^1,\xi^1) - (h(x^*, c) + v^2(x^*,\xi^1))]  \le \epsilon,\\
   A \bar x^1 - b - \delta \in K^1, a.s., \\
 \bbe[\|\delta\|] \leq \epsilon,
\end{array}
\]
where $x^*$ denotes the optimal solution of problem~\ref{first-stage-cp}.
\end{thm}

\begin{proof}
The outer loop of the DSA algorithm is equivalent to the application of Algorithm~2 to
the first stage saddle point problem in~\eqnok{primaldual1}.
Note that for this problem, we have $\tilde v = v^2$. Moreover, by Lemma~\ref{lem_analysis_stage2},
the stochastic subgradients of $v^2$ are computed by the
middle loop with tolerance $\bar \epsilon = 2\epsilon/3$. In view of Corollary~\ref{sto_coro 1} (with $M = M_2$ and $\bar \epsilon = 2\epsilon/3$)
and the definition of $N_1$
in \eqnok{N1}, we conclude that there exist $\delta \in  \bbr^{m^1}$ s.t.
\begin{align*}
\bbe_{\xi^2}[{\rm gap}_\delta(\bar z_N^1)] &\leq  \tfrac{\sqrt{2}\|A^1\| (2\Omega_1^2+\|y_0^1\|^2)}{\sqrt{\alpha_1} N_1}+\tfrac{4\sqrt{3} M_2\Omega_1}{\sqrt{\alpha_1 N_1}} +
\tfrac{2 \epsilon}{3} \le \epsilon,\\
  \bbe_{\xi^2}[\|\delta\|] &\leq \tfrac{2\sqrt{2\alpha_1}\|A^1\| \|y_*^1 - y_0^1\| + 4 \Omega_1 \|A^1\|}{\alpha_1 N_1}
+ \tfrac{2M_2 (\sqrt{6} \|A^1\| + \sqrt{3\alpha_1}) }{\alpha_1 \sqrt{N_1}} + \sqrt{\tfrac{2 \|A^1\|  \epsilon}{N_1 \sqrt{\alpha_1}}} \le \epsilon,
\end{align*}
which together with Lemma~\ref{lem_cone_prog} then imply our result.
\end{proof}

\vgap

We now add a few remarks about the convergence of the DSA algorithm.
Firstly, in view of \eqref{N2} and \eqnok{N3}, $N_2$ and $N_3$ are random variables since they depend on the random variables $\xi^{[2]}$ and $\xi^{[3]}$,
respectively.  The selection of $N_2$ and $N_3$ allows us to remove the boundedness assumptions for a few random variables such as $A_i^2$ and $A_j^3$.
Secondly,  if the random variables appearing in the definition of $N_2$, i.e., $A_i^2$ and $y_{*,i}$, are bounded, we can see from Lemma~\ref{lem_analysis_stage2} and Theorem~\ref{sto_bound_all}
that  the number of random samples $\xi^2$ and $\xi^3$ are given by
\begin{equation} \label{def_samp_complex}
N_1 = \mathcal{O}(1/\epsilon^2) \ \ \mbox{and} \ \ \ N_1\times N_2 = \mathcal{O}(1/\epsilon^4),
\end{equation}
respectively.
It is also possible to obtain an upper bound for $N_2$ and $N_1 \times N_2$ in expectation with respect to $\xi^2$ without assuming the boundedness of  $A_i^2$ and $y_{*,i}$.
Thirdly, it appears that the convergence of the DSA algorithm relies on $y_*^1$, $y_{*,i}^2$, and $y_{*,j}^3$. However, the size of these dual variable can be estimated
by using Lemma~\ref{lemma_dual_bnd}.
and
possibly some tools from random matrix theory~\cite{Rudelson10non-asymptotictheory} to estimate the smallest singular values in case these quantities are not easily computable.

It should be noted that our analysis of DSA focuses on the optimality of the first-stage decisions, and
the decisions we generated for the later stages are mainly used for computing the approximate stochastic
subgradients for the values functions at each stage.
Except for the first stage decision $\bar x^1$,
the performance guarantees (e.g., feasibility and optimality)
 that we can provide for later stages (see Lemma~\ref{lem_analysis_stage3} and~\ref{lem_analysis_stage2}) are
dependent on the sequences of random variables (or scenarios) we generated.
We do not generate
history-dependent policy or suggest a prefixed sequence of decisions for general multi-stage stochastic optimization problems.
However, in some cases such prefixed sequence can still be extracted
from the output of the algorithm. In particular, if one can separate the state and control variables,
then we can use the obtained solutions for the initial state variable and the ones for the control variables in later stages
as a prefixed control policy (see Section~\ref{sec_num}
for an example in portfolio optimization). In general, one possible way to guarantee the feasibility and optimality of the decisions in the later stages
would be to re-run the DSA algorithm in each stage.
More specifically, at the beginning of each stage, we already know the
realization of the random variable at this stage and the decisions from
the previous stage, we can run the DSA algorithm now for a
smaller multi-stage stochastic optimization problem, i.e.,
the number of stages will decrease by $1$ every time we run the algorithm.
One can see that the computational cost for these subsequent runs of the DSA algorithm
will decrease exponentially with respect to the remaining number of stages.
Therefore, the total amount of computational cost
over all these subsequent runs will be in the same order of magnitude
as that for the first run of the DSA method.



\section{Three-stage problems with strongly convex objectives} \label{sec_DSA_sc}
In this section, we show that the complexity of the DSA algorithm can be significantly improved
if the objective functions $h^i$, $i = 1, 2, 3$, are strongly convex.
We will first refine the convergence properties of Algorithm~2 under the strong convexity assumption
about $h(x, c)$ and then use these results to improve the complexity results of the DSA algorithm.

\subsection{Basic tools: inexact primal-dual stochastic approximation under strong convexity}
Our goal in this subsection is to study the convergence properties of Algorithm~2 applied to problem \eqref{primaldual2.2}
under the assumption that $h(x, c)$ is strongly convex, i.e., $\exists \mu_h > 0$ s.t.
\begin{equation} \label{equ_strong0}
h(x_1, c) - h(x_2, c) - \langle h'(x_2, c), x_1 - x_2 \rangle \ge \mu_h P_X(x_2, x_1), \ \forall x_1, x_2 \in X.
\end{equation}

Proposition \ref{sto_prop_strong} below shows the relation
between $(x_{k-1},y_{k-1})$ and $(x_k,y_k)$ after running one step of SPDT when the assumption about $h$ in \eqnok{equ_strong0}
is satisfied.

\begin{prop} \label{sto_prop_strong}
Let $Q$ and $\Delta_k$  be defined in \eqnok{goal_Q} and \eqnok{delta1}, respectively.
For any $1\leq k\leq N$ and $(x,y) \in X \times K_*$, we have
\begin{equation}\label{sto_sc_Q2}
\begin{aligned}
&Q(z_k,z) +\langle A(x_k-x), y_k- y_{k-1} \rangle - \theta_k\langle A(x_{k-1}-x), y_{k-1}- y_{k-2} \rangle\\
 \leq&\ \tau_kP_X(x_{k-1},x)-(\tau_k+\mu_h)P_X(x_k,x)+\tfrac{\eta_k}{2} [\|y_{k-1}-y\|^2-\|y_k-y\|^2] \\
	& -\tfrac{\alpha_X \tau_k}{2}\|x_k-x_{k-1}\|^2- \tfrac{  \eta_k}{2} \|y_{k-1}-y_k\|^2 +\bar \epsilon +(M+\|G_{k-1}\|_*)\|x_k-x_{k-1}\|\\
&+\theta_k\langle A(x_k-x_{k-1}),y_{k-1}-y_{k-2}\rangle+\langle \Delta_{k-1}, x_{k-1}-x\rangle,
\end{aligned}
\end{equation}
\end{prop}

\begin{proof}
Since $h$ is strongly convex, we can rewrite \eqref{sto_x_opt} as
\begin{align*}
    \langle -A_k(x_k-x), \tilde y_k \rangle  + h(x_k,c_k)&-h(x,c_k) + \langle G(x_{k-1},\xi_k), x_k-x\rangle \nonumber\\
    &\leq \tau_kP_X(x_{k-1},x)-(\tau_k+\mu_h)P_X(x_k,x) -\tau_k P_X(x_{k-1},x_k).
\end{align*}
It then follows from \eqnok{goal_Q}, \eqref{sto_V_xk}, \eqref{sto_y_opt} and the above inequality that
\begin{align*}
& Q(z_k,z) +\langle A(x_k-x), y_k-\tilde y_k \rangle \leq \tau_kP_X(x_{k-1},x)-(\tau_k+\mu_h)P_X(x_k,x) -\tau_k P_X(x_{k-1},x_k) \\
	& +\tfrac{\eta_k}{2} [\|y_{k-1}-y\|^2-\|y_k-y\|^2-\|y_{k-1}-y_k\|^2] +(M+\|G_{k-1}\|_*)\|x_k-x_{k-1}\|+\langle \Delta_{k-1}, x_{k-1}-x\rangle +\bar \epsilon .
\end{align*}
Similarly to the proof of \eqnok{sto_Prop 1}, using the above relation, the definition of $\tilde y_k$ in \eqref{def_dual_extra} and the strong convexity of $P$ in \eqref{str_con_PD},
we have \eqref{sto_sc_Q2}.
\end{proof}

\vgap

With the help of Proposition~\ref{sto_prop_strong}, we can provide bounds of two gap functions ${\rm gap}_*(\bar z_N)$ and ${\rm gap}_*\delta(\bar z_N)$
under the strong convexity assumption of $h$.
\begin{thm}\label{thm_sc_h}
Suppose that the parameters $\{\theta_k\}$, $\{w_k\}$, $\{\tau_k\}$ and $\{\eta_k\}$ satisfy \eqnok{sto_cond_1} with
\eqnok{sto_cond_1}.b) replaced by
\begin{equation}\label{sto_sc_cond1}
w_{k}(\mu_h+\tau_{k}) \ge  w_{k+1} \tau_{k+1}, \ k = 1, \ldots, N-1.
\end{equation}
\begin{description}
  \item[a)] For $N\geq 1$, we have
\begin{equation}\label{sto_sc_Q3}
\begin{aligned}
\bbe [{\rm gap}_*(\bar z_N)]\leq & (\tsum_{k=1}^N w_k)^{-1}[2w_1\tau_1 \Omega_X^2+\tfrac{w_1\eta_1}{2} \|y_{0}-y_*\|^2+\tsum_{k=1}^N \tfrac{6M^2w_k}{\alpha_X\tau_k}]+\bar\epsilon.
\end{aligned}
\end{equation}
  \item[b)] If, in addition, $w_1\eta_1 = \ldots = w_N\eta_N$, then
\begin{align}
    & \bbe[{\rm gap}_\delta(\bar z_N)] \leq  (\tsum_{k=1}^N w_k)^{-1}[2w_1\tau_1 \Omega_{X}^2+\tfrac{w_1\eta_1}{2} \|y_0\|^2+\tsum_{k=1}^N \tfrac{6M^2w_k}{2\alpha_X\tau_k}]+\bar\epsilon,\label{sto_sc_g1}\\
   &  \bbe[\|\delta\|] \leq \tfrac{w_1\eta_1}{\tsum_{k=1}^Nw_k}\left[2\|y_*-y_0\|+2\sqrt{\tfrac{\tau_1}{\eta_1}}\Omega_X
  + \sqrt{\tfrac{2}{w_1\eta_1}\tsum_{k=1}^Nw_k\left(\tfrac{6M^2}{\alpha_X\tau_k}+{\bar \epsilon}\right)}\right], \label{sto_sc_delta}\\
  &\bbe[\|y_* - \bar y_N\|^2]\le \|y_* - y_0\|^2 + (\tsum_{k=1}^N w_k)^{-1} \tsum_{k=1}^N \tfrac{2}{\eta_k} \left[2 w_1 \tau_1 \Omega_X^2
+ \tsum_{i=1}^k w_i (\tfrac{6M^2}{\tau_i} +{\bar \epsilon}) \right],\nonumber
\end{align}
where $\delta = (\sum_{k=1}^N w_k)^{-1}[w_1\eta_1(y_0-y_N)]$.
\end{description}
\end{thm}
\begin{proof}
We first show part a) holds. Multiplying both sides of \eqref{sto_sc_Q2} by $w_k$ for every $k\geq 1$, summing up the resulting inequalities over $1\leq k\leq N$,
and using the relations in \eqnok{sto_cond_1} and \eqref{sto_sc_cond1}, we have
\begin{align*}
&\tsum_{k=1}^N w_kQ(z_k,z) \\
 &\leq  \tsum_{k=1}^N [w_k\tau_k P_X(x_{k-1},x)-w_k(\tau_k+\mu_h) P_X(x_k,x)]- \tsum_{k=1}^N\tfrac{\alpha_X w_k \tau_k}{2}\|x_k-x_{k-1}\|^2  \\
& \quad +\tsum_{k=1}^N [\tfrac{w_k\eta_k}{2} \|y_{k-1}-y\|^2- \tfrac{w_k\eta_k}{2} \|y_k-y\|^2] - \tsum_{k=1}^N \tfrac{  w_k\eta_k}{2} \|y_{k-1}-y_k\|^2 \\
& \quad +\tsum_{k=1}^N w_{k-1} \langle A(x_{k-1}-x_k), y_{k-1}- y_{k-2} \rangle+\tsum_{k=1}^N w_k\bar\epsilon +w_N\langle A(x-x_N), y_N- y_{N-1} \rangle\\
& \quad +\tsum_{k=1}^N w_k(M+\|G_{k-1}\|_*)\|x_k-x_{k-1}\|+\tsum_{k=1}^N w_k\langle \Delta_{k-1}, x_{k-1}-x\rangle\\
&\leq  w_1\tau_1 P_X(x_0,x) +\tfrac{w_1\eta_1}{2} \|y_{0}-y\|^2- \tfrac{w_N\eta_N}{2} \|y_{N}-y\|^2\\
&\quad +\tsum_{k=1}^N w_k\bar\epsilon +\tsum_{k=1}^N \tfrac{(M+\|G_{k-1}\|_*)^2w_k}{\alpha_X \tau_k}+\tsum_{k=1}^N w_k\langle \Delta_{k-1}, x_{k-1}-x\rangle\\
    & \quad -w_N(\tau_N+\mu_h) P_X(x_N,x)+w_N\langle A(x-x_N), y_N- y_{N-1} \rangle-\tfrac{ w_N\eta_N}{2}\|y_N-y_{N-1}\|^2\\
 & \leq    w_1\tau_1 P_X(x_0,x)+\tfrac{w_1\eta_1}{2} \|y_{0}-y\|^2- \tfrac{w_N\eta_N}{2} \|y_{N}-y\|^2\\
    & \quad +\tsum_{k=1}^N w_k\bar\epsilon +\tsum_{k=1}^N \tfrac{(M+\|G_{k-1}\|_*)^2w_k}{\alpha_X\tau_k}+\tsum_{k=1}^N w_k\langle \Delta_{k-1}, x_{k-1}-x\rangle ,
\end{align*}
where the last two inequalities follows from similar techniques in the proof of Theorem \ref{sto_Theorem 1}.
Dividing both sides of the above inequality, and using the convexity of $Q$ and the definition of $\bar z_N$, we have
\begin{equation}\label{sto_sc_g3}
\begin{aligned}
\max_{z\in X\times K_*} Q(\bar z_N, z) &\leq  (\tsum_{k=1}^N w_k)^{-1}[w_1\tau_1
\Omega_X^2+\tfrac{w_1\eta_1}{2} \|y_{0}-y\|^2- \tfrac{w_N\eta_N}{2} \|y_{N}-y\|^2\\
    & +\tsum_{k=1}^N w_k\bar\epsilon +\tsum_{k=1}^N \tfrac{(M+\|G_{k-1}\|_*)^2w_k}{\alpha_X\tau_k}+\tsum_{k=1}^N w_k\langle \Delta_{k-1}, x_{k-1}-x\rangle],
\end{aligned}
\end{equation}
which, in view of \eqref{sto_lem1} and \eqnok{def_gap}, then implies
\begin{align*}
{\rm gap}_*(\bar z_N) &\leq  (\tsum_{k=1}^N w_k)^{-1}[2w_1\tau_1 \Omega_X^2+\tfrac{w_1\eta_1}{2} \|y_{0}-y_*\|^2- \tfrac{w_N\eta_N}{2} \|y_{N}-y_*\|^2\\
    & \quad +\tsum_{k=1}^N w_k\epsilon +\tsum_{k=1}^N \tfrac{[\|\Delta_k\|_*^2+2(M+\|G_{k-1}\|_*)^2]w_k}{2\alpha_X\tau_k}+\tsum_{k=1}^N w_k\langle \Delta_{k-1}, x_{k-1}-x_{k-1}^v\rangle].
\end{align*}
Taking expectation w.r.t. $\xi_k$ on both sides of above inequality, and using \eqref{def_to_be_improved} and the fact that $x_{k-1}-x_{k-1}^v$ is independent of $\Delta_{k-1}$, we have
\begin{align*}
\bbe [{\rm gap}_*(\bar z_N)] &\leq  (\tsum_{k=1}^N w_k)^{-1}[2w_1\tau_1 \Omega_X^2+\tfrac{w_1\eta_1}{2} \|y_{0}-y_*\|^2+\tsum_{k=1}^N \tfrac{6M^2w_k}{\alpha_X\tau_k}]+\bar\epsilon.
\end{align*}
The proof of part b) is similar to the one for Theorem~\ref{sto_bounded_thm1}.b) and hence the details are skipped.
\end{proof}

\vgap

In the following two corollaries, we provide two different parameter settings for the selection of $\{w_k\}$, $\{\tau_k\}$ and $\{\eta_k\}$, both of which
can guarantee the convergence of Algorithm~2 in terms of the gap functions $\bbe[{\rm gap}_*(\bar z_N)]$ and $\bbe[{\rm gap}_\delta (\bar z_N)]$.
Moreover, the first one in Corollary~\ref{sc_h} shows that if $M = 0$ and $N$ is properly chosen, then one can ensure the boundedness of $\bbe[\|y_*-\bar y_N\|^2]$,
while the other one in Corollary~\ref{sc_h2} can guarantee the boundedness of $\bbe[\|y_*-\bar y_N\|^2]$ by properly choosing $N$, even under the assumption
that $M > 0$.

\begin{cor}\label{sc_h}
If
\begin{equation}\label{sto_stepsize_con3}
w_k=k,\ \tau_k = \tfrac{k-1}{2}\mu_h\ \mbox{and} \ \eta_k = \tfrac{4\|A\|^2}{k\alpha_X \mu_h},
\end{equation}
then for any $N\geq 1$, we have
  \begin{align}
  \bbe[{\rm gap}_*(\bar z_N)] &\leq \tfrac{8\|A\|^2\|y_0-y_*\|^2}{\alpha_X\mu_h(N+1)N} +\tfrac{24M^2}{\alpha_X\mu_h(N+1)}+ \bar\epsilon,\label{sc_g0}\\
  \bbe[{\rm gap}_\delta (\bar z_N)] &\leq \tfrac{8\|A\|^2\|y_0\|^2}{\alpha_X\mu_h(N+1)N} +\tfrac{24M^2}{\alpha_X\mu_h(N+1)}+ \bar \epsilon,\label{sc_g1}\\
  \bbe[\|\delta\|] &\leq \tfrac{16\|A\|^2\|y_*-y_0\|}{N(N+1)\alpha_X\mu_h}+\tfrac{8\sqrt{6}\|A\|M}{\alpha_X\mu_hN^{3/2}}+
  \tfrac{4\|A\|\sqrt{\bar\epsilon}}{(N+1)\sqrt{\alpha_X\mu_h}},\label{sc_g2}\\
  \bbe[\|y_*-\bar y_N\|^2] &\leq \|y_*-y_0\|^2 + \tfrac{12M^2\alpha_XN}{\|A\|^2} + \tfrac{N(N+1)\alpha_X\mu_h}{2\|A\|^2}\bar \epsilon.\label{sc_g3}
  \end{align}
\end{cor}

\begin{proof}
Clearly, the parameters $w_k$, $\tau_k$ and $\eta_k$ in \eqref{sto_stepsize_con3} satisfy \eqnok{sto_cond_1} with
\eqnok{sto_cond_1}.b) replaced by \eqref{sto_sc_cond1}.
It then follows from Theorem \ref{thm_sc_h} and \eqref{sto_stepsize_con3} that
$$
\begin{array}{lll}
\bbe[{\rm gap}_*(\bar z_N)] & \leq& \tfrac{2}{N(N+1)}\left[\tfrac{4\|A\|^2\|y_*-y_0\|^2}{\alpha_X\mu_h} + \tfrac{12M^2N}{\alpha_X\mu_h}\right]+\bar \epsilon\\
& \leq& \tfrac{8\|A\|^2\|y_0-y_*\|^2}{\alpha_X\mu_h(N+1)N} +\tfrac{24M^2}{\alpha_X\mu_h(N+1)}+ \bar\epsilon,\\
\bbe[{\rm gap}_\delta (\bar z_N)]  & \leq& \tfrac{8\|A\|^2\|y_0\|^2}{\alpha_X\mu_h(N+1)N} +\tfrac{24M^2}{\alpha_X\mu_h(N+1)}+ \bar \epsilon,\\
\bbe[\|\delta\|] &\leq& \tfrac{8\|A\|^2}{\alpha_X\mu_hN(N+1)}\left[2\|y_*-y_0\|+\sqrt{\tfrac{\alpha_X\mu_h}{2\|A\|^2}(\tfrac{6M^2}{\alpha_X}2N+\tfrac{N(N+1)}{2}\bar\epsilon)}\right]\\
& \leq& \tfrac{16\|A\|^2\|y_*-y_0\|}{N(N+1)\alpha_X\mu_h}+\tfrac{8\sqrt{6}\|A\|M}{\alpha_X\mu_hN^{3/2}}+
  \tfrac{4\|A\|\sqrt{\bar\epsilon}}{(N+1)\sqrt{\alpha_X\mu_h}},\\
\bbe[\|y_*-\bar y_N\|^2]  & \leq& \|y_*-y_0\|^2 + \tfrac{2}{N(N+1)}\tsum_{k=1}^N\tfrac{k\alpha_X\mu_h}{2\|A\|^2}\left(2N\tfrac{12M^2}{\mu_h}+\tfrac{N(N+1)}{2}\bar\epsilon\right)\\
&=&\|y_*-y_0\|^2 + \tfrac{12M^2\alpha_XN}{\|A\|^2} + \tfrac{N(N+1)\alpha_X\mu_h}{2\|A\|^2}\bar \epsilon.
\end{array}
$$

\end{proof}
\vgap

\begin{cor}\label{sc_h2}
If
\begin{equation}\label{sto_stepsize_con4}
w_k=k,\ \tau_k = \tfrac{k-1}{2}\mu_h \ \ \mbox{and} \ \ \eta_k = \tfrac{4\|A\|^2N}{k\alpha_X \mu_h},
\end{equation}
then for any $N\geq 1$, we have
  \begin{align}
  \bbe[{\rm gap}_*(\bar z_N)] &\leq \tfrac{8\|A\|^2\|y_0-y_*\|^2 +24M^2}{\alpha_X\mu_h(N+1)}+ \bar\epsilon,\label{sc_g4}\\
  \bbe[{\rm gap}_\delta (\bar z_N)] &\leq \tfrac{8\|A\|^2\|y_0\|^2 +24M^2}{\alpha_X\mu_h(N+1)}+ \bar \epsilon,\label{sc_g5}\\
  \bbe[\|\delta\|] &\leq \tfrac{16\|A\|^2\|y_*-y_0\|}{(N+1)\alpha_X\mu_h + 16\sqrt{3}\|A\|M}+
  \tfrac{4\|A\|\sqrt{\bar\epsilon}}{\sqrt{(N+1)\alpha_X\mu_h}},\label{sc_g6}\\
  \bbe[\|y_*-\bar y_N\|^2] &\leq \|y_*-y_0\|^2 + \tfrac{24M^2\alpha_X}{\|A\|^2} + \tfrac{(N+1)\alpha_X\mu_h}{2\|A\|^2}\bar \epsilon.\label{sc_g7}
  \end{align}
\end{cor}
\begin{proof}
The proofs of \eqref{sc_g4}-\eqref{sc_g7} are similar to Corollary \ref{sc_h} and hence the details are skipped.
\end{proof}
\vgap

%

\subsection{Convergence analysis for DSA under strong convexity}
Our goal in this subsection is to establish the complexity of the DSA algorithm for solving problem~\ref{first-stage-cp} under the strong
convex assumption about $h^i$, $i =1, 2, 3$, i.e., $\exists \mu_i >0$ s.t.
\begin{equation} \label{equ_strong1}
h^i(x_1, c) - h^i(x_2, c) - \langle (h^i)'(x_2, c), x_1 - x_2 \rangle \ge \mu_i P_{X^i}(x_2, x_1), \ \forall x_1, x_2 \in X^i.
\end{equation}

We describe some convergence properties for the innermost and middle loop of the DSA algorithm under
the strong convexity assumptions in \eqnok{equ_strong1} in Lemma~\ref{lem_analysis_sc_stage3} and \ref{lem_analysis_sc_stage2}, respectively.
The proofs for these results are similar to those for Lemma~\ref{lem_analysis_stage3} and \ref{lem_analysis_stage2}.

Lemma~\ref{lem_analysis_sc_stage3} below describes the convergence properties for the innermost loop of the DSA algorithm.
\begin{lem} \label{lem_analysis_sc_stage3}
If the parameters $\{w_k^3\}$, $\{\tau_k^3\}$ and $\{\eta_k^3\}$ are set to \eqref{sto_stepsize_con3} (with $M = 0$ and $A = A_j^3$) and
\begin{equation} \label{sc_N3}
N_3 \equiv N_{3,j} :=  \tfrac{2\sqrt{6}\|A_j^3\|\|y^3_{*,j}-y^3_{0}\|}{\sqrt{\alpha_3\mu_3\epsilon}},
\end{equation}
then $B_j^3 \bar y_j^3$ is a stochastic $(\epsilon/3)$-subgradient of the value function $v^3$ at $x_{j-1}^2$.
Moreover, there exists a constant $M_3 \ge 0$ such that $\|v^3(x_1,\xi^{[2]}) - v^3 (x_2,\xi^{[2]}) \| \le M_3 \|x_1 - x_2\|, \forall x_1, x_2 \in X^3$ and
\begin{equation} \label{def_sc_M3}
\bbe[\|B_j^3 \bar y_j^3\|_*^2 | \xi^{[2]}] \le M_3.
\end{equation}
In addition, there exists a vector $\delta \in \bbr^{m^3}$ s.t.
\[
  \begin{array}{l}
  \bbe[h^3(\bar x^3, c^3) - V^2(\bar x^2, \xi^{[3]})  ] \le \epsilon/3,\\
   A^3 \bar x^3 - B^3\bar x^2 - b^3 - \delta \in K^3 \ {\rm a.s.},\\
    \bbe[\|\delta\| | \xi^{[2]}] \le \epsilon/3.
\end{array}
\]
\end{lem}

\begin{proof}
In view of Corollary~\ref{sc_h} (with $M = 0$ and $\bar \epsilon = 0$)
and the definition of $N_3$
in \eqnok{sc_N3}, we have
\[
\bbe[{\rm gap}_*(\bar z_j^3) |\xi^{[2]}]\ \le \tfrac{8\|A_j^3\|^2\|y^3_0-y^3_*\|^2}{\alpha_3\mu_3(N_3+1)N_3}  \le \tfrac{\epsilon}{3}.
\]
This observation, in view of Lemma~\ref{epsi_subg}, then implies that
$B_j^3 \bar y_j^3$ is a stochastic $(\epsilon/3)$-subgradient of $v^3$ at $x_{j-1}^2$.
Moreover, it follows from \eqref{sc_g3} (with $M = 0$ and ${\bar \epsilon}=0$) that
$
\bbe[\|y_{*,j}^3 - \bar y_{j}^3\|| \xi^{[2]}] \le \|y_{*,j}^3 - y_0^3\| .
$
This inequality, in view of the selection of $N_3$ in \eqnok{sc_N3}, the assumption that
$y_{*,j}^3$ is well-defined, and \eqnok{boundedB3}, then implies the latter part of our result.
The techniques are similar to the proof of Lemma~\ref{lem_analysis_stage3} and the details are
skipped.
\end{proof}

\vgap

Lemma~\ref{lem_analysis_sc_stage3} below describes the convergence properties for the middle loop of the DSA algorithm.
\begin{lem} \label{lem_analysis_sc_stage2}
Assume that the parameters for the innermost loop are set according to Lemma~\ref{lem_analysis_sc_stage3}.
If the parameters $\{w_j^2\}$, $\{\tau_j^2\}$ and $\{\eta_j^2\}$ for the middle loop are set to \eqref{sto_stepsize_con4} (with $M = M_3$ and $A = A_i^2$) and
\begin{equation} \label{sc_N2}
N_2 \equiv N_{2,i} :=  \tfrac{24\|A_i^2\|^2\|y^2_0-y^2_{*,i}\|^2 +72M_3^2}{\alpha_2\mu_2\epsilon},
\end{equation}
then $B_i^2 \bar y_i^2$ is a stochastic $(2\epsilon/3)$-subgradient of the value function $v^2$ at $x_{i-1}^1$.
Moreover, there exists a constant $M_2 \ge 0$ such that $\|v^2(x_1,\xi^{[1]}) - v^2 (x_2,\xi^{[1]}) \| \le M_2 \|x_1 - x_2\|, \forall x_1, x_2 \in X^2$ and
\begin{equation} \label{def_sc_M2}
\bbe[\|B_i^2 \bar y_i^2\|_*^2|\xi^{[1]}] \le M_2.
\end{equation}
In addition, there exists a vector $\delta \in \bbr^{m^2}$ s.t.
\[
  \begin{array}{l}
  \bbe[h^2(\bar x^2, c^2) + v^3(\bar x^2|\xi^2) - V^2(\bar x^{1}, \xi^{[2]}) |\xi^{[1]} ] \le 2\epsilon/3,\\
   A^2 \bar x^2 - B^2\bar x^1 - b^2 - \delta \in K^2 \ {\rm a.s.},\\
    \bbe[\|\delta\| | \xi^{[1]}] \le 2\epsilon/3.
\end{array}
\]
\end{lem}

\begin{proof}
By Lemma~\ref{lem_analysis_sc_stage3},
the stochastic subgradients of $v^3$ are computed by the
innermost loop with tolerance $\bar \epsilon = \epsilon/3$. In view of Corollary~\ref{sc_h2} (with $M = M_3$ and $\bar \epsilon = \epsilon/3$)
and the definition of $N_2$
in \eqnok{sc_N2}, we have
\[
\bbe[{\rm gap}_*(\bar z_i^2)|\xi^{[1]}]  \le \tfrac{8\|A_i^2\|^2\|y^2_0-y^2_{*,i}\|^2 +24M_3^2}{\alpha_2\mu_2(N_2+1)}+ \bar\epsilon\le \tfrac{2\epsilon}{3}.
\]
This observation, in view of Lemma~\ref{epsi_subg}, then implies that
$B_i^2 \bar y_i^2$ is a stochastic $(2\epsilon/3)$-subgradient $v^2$ at $x_{i-1}^1$.
Moreover, it follows from \eqnok{sc_g7} (with $M = M_3$ and ${\bar \epsilon}=\epsilon/3$) that
\[
\bbe[\|y_{*,i}^2 - \bar y_i^2\|^2| \xi^{[1]}] \le \|y_{*,i}^2 - y_0^2\|^2 +  \tfrac{24M_3^2\alpha_2}{\|A_i^2\|^2} + \tfrac{(N_2+1)\alpha_2\mu_2}{6\|A^2_i\|^2} \epsilon.
\]
This inequality, in view of the selection of $N_2$ in \eqnok{sc_N2}, the assumption that
$y_{*,i}^2$ is well-defined, and \eqnok{boundedB3}, then implies the latter part of our result.
The techniques are similar to the proof of Lemma~\ref{lem_analysis_stage2} and the details are
skipped.
\end{proof}

\vgap

We are now ready to state the main convergence properties of the DSA algorithm for solving strongly convex
three-stage stochastic optimization problems.
\begin{thm}\label{sto_sc_bound_all}
Suppose that the parameters for the innermost and middle loop in the DSA algorithm are set according to Lemma~\ref{lem_analysis_sc_stage3}
and Lemma~\ref{lem_analysis_sc_stage2}, respectively. If
the parameters $\{w_i\}$, $\{\tau_i\}$ and $\{\eta_i\}$ for the outer loop are set to \eqnok{sto_stepsize_con3} (with $M = M_2$ and $A = A^1$) and
\begin{equation} \label{sc_N1}
\begin{array}{l}
N_1 := \max \left\{\tfrac{4\sqrt{3}\|A^1\|\|y_0^1\|}{\sqrt{\alpha_1\mu_1\epsilon}}+\tfrac{4(6M_2)^2}{\alpha_1\mu_1\epsilon}, \right.\\
\quad \quad \left. \tfrac{4\sqrt{3}\|A^1\|(\sqrt{\|y_*^1-y_0^1\|}+\sqrt{2})}{\sqrt{\alpha_1\mu_1\epsilon}} + \left(\tfrac{24\sqrt{6}\|A^1\|M_2}{\alpha_1\mu_1\epsilon}\right)^{2/3}
\right\},
\end{array}
\end{equation}
then we will find a solution $\bar x^1 \in X^1$ and a vector $\delta \in \bbr^{m^1}$ s.t.
\[
  \begin{array}{l}
  \bbe[h(\bar x^1, c^1) + v^2(\bar x^1,\xi^1) - (h(x^*, c^1) + v^2(x^*,\xi^1)) ] \le \epsilon,\\
   A \bar x^1 - b - \delta \in K^1, a.s., \\
 \bbe[\|\delta\|] \leq \epsilon,
\end{array}
\]
where $x^*$ denotes the optimal solution of problem~\ref{first-stage-cp}.
\end{thm}

\begin{proof}
By Lemma~\ref{lem_analysis_sc_stage2},
the stochastic subgradients of $v^2$ are computed by the
middle loop with tolerance $\bar \epsilon = 2\epsilon/3$. In view of Corollary~\ref{sc_h} (with $M = M_2$ and $\bar \epsilon = 2\epsilon/3$)
and the definition of $N_1$
in \eqnok{sc_N1}, we conclude that there exist $\delta \in  \bbr^{m^1}$ s.t.
\begin{align*}
\bbe[{\rm gap}_\delta(\bar z_N^1)] &\leq  \tfrac{8\|A^1\|^2\|y^1_0\|^2}{\alpha_1\mu_1(N_1+1)N_1} +\tfrac{24M_2^2}{\alpha_1\mu_1(N_1+1)}+ \tfrac{2 \epsilon}{3}\leq \epsilon,\\
  \bbe[\|\delta\|] &\leq \tfrac{16\|A^1\|^2\|y^1_*-y^1_0\|}{N_1(N_1+1)\alpha_1\mu_1}+\tfrac{8\sqrt{6}\|A^1\|M_2}{\alpha_1\mu_1N_1^{3/2}}+
  \tfrac{4\|A^1\|\sqrt{2\epsilon}}{(N_1+1)\sqrt{3\alpha_1\mu_1}}\le \epsilon,
\end{align*}
which together with Lemma~\ref{lem_cone_prog} then imply our result.
\end{proof}

\vgap

In view of Lemma~\ref{lem_analysis_sc_stage2} and Theorem~\ref{sto_sc_bound_all},
the number of random samples $\xi_2$ and $\xi_3$ will be bounded by $N_1$ and $N_1\times N_2$, i.e., $\mathcal{O}(1/\epsilon)$ and $\mathcal{O}(1/\epsilon^2)$, respectively,
under the assumption that the random variables appearing in the definition of $N_2$ (i.e., $A^2_i$ and $y^2_{*,i}$) are bounded.

\section{DSA for general multi-stage stochastic optimization}
In this section, we consider a multi-stage stochastic optimization problem given by
\begin{equation} \label{multi-stage prob}
\begin{array}{ll}
\min \ \left\{ h^1(x^1,c^1)+ v^{2}(x^1,\xi^1) \right\} \\
\text{ s.t.} \ \ A^1 x^1 - b^1 \in K^1,\\
\quad \quad \quad x^1 \in X^1,
\end{array}
\end{equation}
where the value factions $v^t$, $t =2, \ldots, T$, are recursively defined by
\begin{equation}\label{Defi_sto_V_m}
\begin{array}{lll}
 v^t(x^{t-1},\xi^{[t-1]}) &:=& F^{t-1} (x^{t-1},p^{t-1}) + \bbe [V^t(x^{t-1},\xi^{[t]})|\xi^{[t-1]}], \ \ t = 2, \ldots, T-1,\\
V^t(x^{t-1},\xi^{[t]}) &:= & \min \ \left\{ h^t(x^t,c^t)+ v^{t+1}(x^t) \right\}\\
 &&\ \text{  s.t.} \ \ A^tx^t - b^t -B^tx^{t-1} \in K^t,\\
 & & \quad \quad \quad x^t\in X^t,
 \end{array}
 \end{equation}
 and
\begin{equation}\label{Defi_sto_V1_m}
 \begin{array}{lll}
  v^T(x^{T-1},\xi^{[T-1]}) &:=& \bbe_{\xi^T} [V^T(x^{T-1},\xi^{[T]})|\xi^{[T-1]}],\\
V^T(x^{T-1},\xi^{[T]}) &:=&  \min \ h^T(x^T,c^T)\\
  &&\text{ s.t.} \ \ A^T x^T - b^T - B^T x^{T-1} \in K^T, \\
  & &\quad \quad \quad x^T\in X^T.
\end{array}
\end{equation}
Here $\xi^t := (A^t,b^t,B^t,c^t,p^t)$ are random variables, $h^t(\cdot,c^t)$ are relatively simple functions, $F^t(\cdot,p^t)$ are general (not necessarily simple)
Lipschitz continuous convex functions and $K^t$ are convex cones, $\forall t=1,\ldots,T$. We also assume that one can compute the subgradient $F'(x^t,p^t)$ of
function $F^t(x^t,p^t)$ at any point $x^t\in X^t$ for a given parameter $p^t$.

Problem~\eqnok{multi-stage prob} is more general than problem \eqref{3-stage prob} (or equivalently problem~\eqnok{first-stage-cp}) in the
following sense. First, we are dealing with a more complicated multi-stage stochastic optimization problem where the number of stages $T$ \eqnok{multi-stage prob} can be greater than three.
Second, the value function $ v^t(x^{t-1}, \xi^{[t-1]})$ in \eqnok{Defi_sto_V_m}
is defined as the summation of  $F^{t-1} (x^{t-1},p^{t-1})$ and $ \bbe[V^t(x^{t-1},\xi^{[t]})|\xi^{[t-1]}]$,
where $F^{t-1}$ is not necessarily simple.
We intend to generalize the DSA algorithm in Sections~\ref{sec_DSA} and \ref{sec_DSA_sc} for solving problem~\eqnok{multi-stage prob}.
More specifically, we show how to compute a stochastic $\epsilon$-subgradient of $v^{t+1}$
at $x^t$, $t= 1, \ldots, T-2$, in a recursive manner until we obtain the $\epsilon$-subgradient of $v^T$ at $x^{T-1}$.

We are now ready to formally state the DSA algorithm for solving the multi-stage stochastic optimization problem in \eqnok{multi-stage prob}.
Observe that the following notations will be used in the algorithm:
\begin{itemize}
  \item $N_t$ is the number of iterations for stage $t$ subproblem and $k_t$ is the corresponding index, i.e., $k_t = 1,\ldots, N_t$.
  \item $\xi_{k_{t-1}}^t= (A_{k_{t-1}}^t,b_{k_{t-1}}^t,B_{k_{t-1}}^t,c_{k_{t-1}}^t,p_{k_{t-1}}^t)$ is the $k_{t-1}$ th
  random scenarios in stage $t$ subproblem, $(x_{k_t}^t, y_{k_t}^t)$ are the $k_t$ th iterates in stage $t$ subproblem.
  \item For simplicity, we denote $\xi_{k_{t-1}}^t$ as $\xi_k^t$, $(x_{k_t}^t, y_{k_t}^t)$ as $(x_{k}^t, y_{k}^t)$.
\end{itemize}

\begin{algorithm}[H]
\caption{DSA for multi-stage stochastic programs}
\begin{algorithmic}
\State {\bf Input:} initial points $\{x_0^t\}$, $k_t=1,\forall t,$ iteration number $N_t$ and stepsize strategy $\{w_k\}$.
\State Start with procedure ${\rm DSA}(1,0)$.
\State {\bf procedure:} ${\rm DSA}(t,u)$
\For {$k_t =1,\ldots, N_t$}
\If {$t < T$}
\State Generate random scenarios $\xi_k^{t+1}$.
\State $(\bar x^{t+1},\bar y^{t+1})= {\rm DSA}(t+1,x_{k}^t)$ and $G(x_{k-1}^t,\xi_{k}^{t+1})= (B_{k}^{t+1})^T\bar y^{t+1}$.
\Else
\State $G(x_{k-1}^T,\xi_{k}^{T+1}) = 0$.
\EndIf
\State $(x_k^t, y_k^t)={\rm SPDT}(x_{k-1}^t,y_{k-1}^t,y_{k-2}^t,G(x_{k-1}^t,\xi_{k}^{t+1}),u,\xi_{k-1}^{t},h^t,X^t,K_*^t,\theta_k^t,\tau_k^t,\eta_k^t)$.
\EndFor
\State {\bf return:} $\bar z^t = \tsum_{k=1}^{N_t} w_kz_k^t / \tsum_{k=1}^{N_t}w_k$.
\end{algorithmic}
\end{algorithm}


In order to show the convergence of the above DSA algorithm, we
need the following assumption on the boundedness of the operators $B^t$:
\begin{equation} \label{boundedBt}
\|B^t\| \le {\cal B}_t,\ \forall t=2,\cdots, T.
\end{equation}

Lemma~\ref{lem_analysis_multistage1} below establishes some convergence properties of the DSA algorithm
for solving the last stage problem.

\begin{lem} \label{lem_analysis_multistage1}
Suppose that the algorithmic parameters in the DSA algorithm applied to problem~\ref{multi-stage prob}
are chosen as follows.
\begin{description}
  \item[a)] For a general convex problem, $\{w_k^T\}$, $\{\tau_k^T\}$ and $\{\eta_k^T\}$ are set to \eqref{step1} (with $M = 0$ and $A = A_k^T$) and
\begin{equation} \label{N_T}
N_T \equiv N_{T,k} :=  \tfrac{T\sqrt{2}\|A_k^T\| [2(\Omega_T)^2+\|y^T_{*,k}-y^T_{0}\|^2]}{\sqrt{\alpha_T} \epsilon}.
\end{equation}
  \item[b)] Under the strongly convex assumption~\eqref{equ_strong1}, $\{w_k^T\}$, $\{\tau_k^T\}$ and $\{\eta_k^T\}$ are set to \eqref{sto_stepsize_con3} (with $M = 0$ and $A = A_k^T$) and
\begin{equation} \label{sc_N_T}
N_T \equiv N_{T,k} :=  \tfrac{\sqrt{8T}\|A_k^T\|\|y^T_{*,k}-y^T_{0}\|}{\sqrt{\alpha_T\mu_T\epsilon}}.
\end{equation}
\end{description}
Then $B_k^T \bar y_k^T$ is a stochastic $(\epsilon/T)$-subgradient of the value function $v^T$ at $x_{k-1}^{T-1}$.
Moreover, there exists a constant $M_T \ge 0$ such that $\|v^T(x_1,\xi^{[T-1]}) - v^T (x_2,\xi^{[T-1]}) \| \le M_T \|x_1 - x_2\|, \forall x_1, x_2 \in X^T$ and
\begin{equation} \label{def_M_T}
\bbe_{\xi^T}[\|B_k^T \bar y_k^T\|_*^2] \le M_T.
\end{equation}
\end{lem}

\begin{proof}
The innermost loop of the DSA algorithm is equivalent to the application of Algorithm~2 to
the last stage saddle point problem in~\eqnok{primaldual3}.
Note that for this problem, we do not have any subsequent stages and hence $\tilde v = 0$. In other words,
the subgradients of $\tilde v$ are exact. To show part a), in view of Corollary~\ref{sto_coro 1} (with $M = 0$ and $\bar \epsilon = 0$)
and the definition of $N_T$
in \eqnok{N_T}, we have
\[
\bbe[{\rm gap}_*(\bar z_k^T) | \xi^{[T-1]}] \le \tfrac{\sqrt{2}\|A_k^T\| [2(\Omega_T)^2+\|y^T_*-y^T_0\|^2]}{\sqrt{\alpha_T} N_T} \le \tfrac{\epsilon}{T}.
\]
This observation, in view of Lemma~\ref{epsi_subg}, then implies that
$B_k^T \bar y_k^T$ is a stochastic $(\epsilon/T)$-subgradient of $v^T$ at $x_{j-1}^{T-1}$.
Moreover, it follows from \eqnok{g_bound2} (with $M = 0$ and ${\bar \epsilon}=0$) that
\[
\bbe[\|y_{*,k}^T - \bar y_{k}^T\|^2| \xi^{[T-1]}] \le \|y_{*,k}^T - y_0^T\|^2 + 4 (\Omega_T)^2 + \tfrac{(N_T+1) \epsilon}{2}.
\]
This inequality, in view of the selection of $N_T$ in \eqnok{N_T}, the assumption that
$y_{*,k}^T$ is well-defined, and \eqnok{boundedBt}, then implies the latter part of our result.
Similarly, the result in \eqref{sc_N_T} follows from Corollary~\ref{sc_h} (with $M = 0$ and $\bar \epsilon = 0$)
and the definition of $N_T$
in \eqnok{sc_N_T}.
\end{proof}

\vgap

We show in Lemma~\ref{lem_analysis_multistage2} some convergence properties of
the middle loops of the DSA algorithm.

\begin{lem} \label{lem_analysis_multistage2}
Assume that the parameters for the innermost loop are set according to Lemma~\ref{lem_analysis_multistage1}.
Moreover, suppose that the algorithmic parameters for the middle loops are chosen as follows.
\begin{description}
  \item[a)] For general convex problem, the parameters $\{w_k^t\}$, $\{\tau_k^t\}$ and $\{\eta_k^t\}$ for the middle loops ($t=2,\ldots,T-1$)
  are set to \eqref{step2} (with $M = M_{t+1}$ and $A = A_k^t$) and
\begin{equation} \label{N_t}
N_t \equiv N_{t,k} :=  \left(\tfrac{ 4 \sqrt{2}T \|A^t_k\| \Omega_t}{\sqrt{\alpha_t} \epsilon} \right)^\frac{2}{3}
     + \left[\tfrac{2T\left(\|A^t_k\| \|y_{*,k}^t - y_0^t\|^2 + 4 \sqrt{3} M_{t+1} \Omega_t\right)}{\sqrt{\alpha_t} \epsilon}\right]^2.
\end{equation}
  \item[b)] Under strongly convex assumption~\eqref{equ_strong1}, the parameters $\{w_k^t\}$, $\{\tau_k^t\}$ and $\{\eta_k^t\}$ for the
  middle loops ($t=2,\ldots,T-1$) are set to \eqref{sto_stepsize_con4} (with $M = M_{t+1}$ and $A = A_k^t$) and
      \begin{equation} \label{sc_N_t}
N_t \equiv N_{t,k} := \tfrac{8T\|A_k^t\|^2\|y^t_0-y^t_{*,k}\|^2 +24TM_{t+1}^2}{\alpha_t\mu_t\epsilon}.
\end{equation}
\end{description}
Then $B_k^t \bar y_k^t$ is a stochastic $((T+1-t)\epsilon/T)$-subgradient of the value function $v^t$ at $x_{k-1}^{t-1}$.
Moreover, there exists a constant $M_t \ge 0$ such that $\|v^t(x_1,\xi^{[t-1]}) - v^t (x_2,\xi^{[t-1]}) \| \le M_t \|x_1 - x_2\|, \forall x_1, x_2 \in X^t$ and
\begin{equation} \label{def_M_t}
\bbe[\|B_k^t \bar y_k^t\|_*^2| \xi^{[t-1]}] \le M_t.
\end{equation}
\end{lem}

\begin{proof}
The middle loops ($t=2,\ldots,T-1$) of the DSA algorithm applied to multistage stochastic optimization is equivalent to the application of Algorithm~2 to
the second stage saddle point problem in~\eqnok{primaldual2}.
Note that for this problem, we have $\tilde v = v^{t+1}$. Moreover, by Lemma~\ref{lem_analysis_multistage1},
the stochastic subgradients of $v^T$ are computed by the
innermost loop with tolerance $\bar \epsilon = \epsilon/T$. To show part a), in view of Corollary~\ref{cor_bnd_dual} (with $M = M_{t+1}$ and $\bar \epsilon = (T-t)\epsilon/T$)
and the definition of $N_t$
in \eqnok{N_t}, we have
\[
\bbe[{\rm gap}_*(\bar z_k^t) | \xi^{[t-1]} ] \le \tfrac{2 \sqrt{2} \|A^t_k\| \Omega_t}{N_t \sqrt{\alpha_t N_t}}
     + \tfrac{\|A^t_k\| \|y_{*,k}^t - y_0^t\|^2 + 4 \sqrt{3} M_{t+1} \Omega_t}{\sqrt{\alpha_t N_t}} +{\bar \epsilon}\le \tfrac{(T+1-t)\epsilon}{T}.
\]
This observation, in view of Lemma~\ref{epsi_subg}, then implies that
$B_k^t \bar y_k^t$ is a stochastic $((T+1-t)\epsilon/T)$-subgradient $v^{t}$ at $x_{k-1}^{t-1}$.
Moreover, it follows from \eqnok{sto_cor2_d} (with $M = M_{t+1}$ and ${\bar \epsilon}=(T-t)\epsilon/T$) that
\[
\bbe[\|y_{*,k}^t - \bar y_k^t\|^2| \xi^{[t-1]}] \le \|y_{*,k}^t - y_0^t\|^2 + \tfrac{2 \Omega_t^2}{N_t} + \tfrac{\sqrt{6} (1+\alpha_t) M_{t+1} \Omega_t}{\|A^t_k\|}
+ \tfrac{\sqrt{\alpha_t N_t} { \epsilon}}{3\sqrt{2} \|A^t_k\|}.
\]
This inequality, in view of the selection of $N_t$ in \eqnok{N_t}, the assumption that
$y_{*,k}^t$ is well-defined, and \eqnok{boundedBt}, then implies the latter part of our result. Similarly, in view of Corollary~\ref{sc_h2}, we have part b).
\end{proof}

\vgap

We are now ready to establish the main convergence properties of the DSA algorithm for solving general multi-stage stochastic optimization problems with $T \ge 3$.
\begin{thm}\label{sto_multistage_all}
Suppose that the parameters for the inner loops in the DSA algorithm are set according to Lemma~\ref{lem_analysis_multistage1}
and Lemma~\ref{lem_analysis_multistage2}. Moreover, assume that the algorithmic parameters
in the outer loop of the DSA algorithm are chosen as follows.
\begin{description}
  \item[a)] For general convex problem, the parameters $\{w_k\}$, $\{\tau_k\}$ and $\{\eta_k\}$ for the outer loop are set to \eqnok{step1} (with $M = M_2$ and $A = A^1$) and
\begin{equation} \label{N_1}
\begin{array}{l}
N_1 := \max \left\{ \tfrac{2\sqrt{2}T\|A^1\| [2(\Omega_1)^2+\|y_0^1\|^2]}{\sqrt{\alpha_1} \epsilon}+ \left(\tfrac{8\sqrt{3} T M_2\Omega_1}{\sqrt{\alpha_1} \epsilon}\right)^2, \right.\\
\quad \quad \left. \tfrac{6T\|A^1\| (\sqrt{2\alpha_1} \|y_*^1 - y_0^1\| + 2 \Omega_1) + 27(T-1) \sqrt{\alpha_1}\|A^1\|}{\alpha_1 T\epsilon} + \left( \tfrac{6 \sqrt{3} M_2 (\sqrt{2} \|A^1\| + \sqrt{\alpha_1}) }{\alpha_1 \epsilon}\right)^2
\right\}.
\end{array}
\end{equation}
  \item[b)] Under strongly convex assumption~\eqref{equ_strong1}, the parameters $\{w_k\}$, $\{\tau_k\}$ and $\{\eta_k\}$ for the outer loop are
  set to \eqnok{sto_stepsize_con3} (with $M = M_2$ and $A = A^1$) and
\begin{equation} \label{sc_N_1}
\begin{array}{l}
N_1 := \max \left\{\tfrac{4\sqrt{T}\|A^1\|\|y_0^1\|}{\sqrt{\alpha_1\mu_1\epsilon}}+\tfrac{24TM_2^2}{\alpha_1\mu_1\epsilon}, \right.\\
\quad \quad \left. \tfrac{4\sqrt{3}\|A^1\|\sqrt{\|y_*^1-y_0^1\|}}{\sqrt{\alpha_1\mu_1\epsilon}} + \left(\tfrac{24\sqrt{6}\|A^1\|M_2}{\alpha_1\mu_1\epsilon}\right)^{2/3} +\tfrac{12\|A^1\|\sqrt{T-1}}{\sqrt{\alpha_1\mu_1T\epsilon}}
\right\}.
\end{array}
\end{equation}
\end{description}
Then we will find a solution $\bar x^1 \in X^1$ and a vector $\delta \in \bbr^{m^1}$ s.t.
\[
  \begin{array}{l}
  \bbe[h(\bar x^1, c) + v^2(\bar x^1,\xi^{1}) - (h(x^*, c) + v^2(x^*,\xi^{1})) ] \le \epsilon,\\
   A \bar x^1 - b - \delta \in K^1, a.s., \\
 \bbe[\|\delta\|] \leq \epsilon,
\end{array}
\]
where $x^*$ denotes the optimal solution of problem~\ref{first-stage-cp}.
\end{thm}

\begin{proof}
The outer loop of the DSA algorithm is equivalent to the application of Algorithm~2 to
the first stage saddle point problem in~\eqnok{primaldual1}.
Note that for this problem, we have $\tilde v = v^2$. Moreover, by Lemma~\ref{lem_analysis_multistage2},
the stochastic subgradients of $v^2$ are computed by the
middle loop with tolerance $\bar \epsilon = (T-1)\epsilon/T$. To show part a), in view of Corollary~\ref{sto_coro 1} (with $M = M_2$ and $\bar \epsilon = (T-1)\epsilon/T$)
and the definition of $N_1$
in \eqnok{N_1}, we conclude that there exist $\delta \in  \bbr^{m^1}$ s.t.
\begin{align*}
\bbe[{\rm gap}_\delta(\bar z_N^1)] &\leq  \tfrac{\sqrt{2}\|A^1\| (2\Omega_1^2+\|y_0^1\|^2)}{\sqrt{\alpha_1} N_1}+\tfrac{4\sqrt{3} M_2\Omega_1}{\sqrt{\alpha_1 N_1}} +
\tfrac{(T-1) \epsilon}{T} \le \epsilon,\\
  \bbe[\|\delta\|] &\leq \tfrac{2\sqrt{2\alpha_1}\|A^1\| \|y_*^1 - y_0^1\| + 4 \Omega_1 \|A^1\|}{\alpha_1 N_1}
+ \tfrac{2M_2 (\sqrt{6} \|A^1\| + \sqrt{3\alpha_1}) }{\alpha_1 \sqrt{N_1}} + \sqrt{\tfrac{3 \|A^1\| (T-1) \epsilon}{N_1 T\sqrt{\alpha_1}}} \le \epsilon,
\end{align*}
which together with Lemma~\ref{lem_cone_prog} then imply our result. Similarly, in view of Corollary~\ref{sc_h}, we have part b).
\end{proof}

\vgap

In view of the results stated in Lemma~\ref{lem_analysis_multistage1}, Lemma~\ref{lem_analysis_multistage2} and Theorem~\ref{sto_multistage_all},
the total number of scenarios required to find an $\epsilon$-solution of \eqnok{multi-stage prob} is given by $N_2 \times N_3 \times \ldots N_T$,
and hence will grow exponentially with respect to $T$, no matter the objective functions are strongly convex or not.
These sampling complexity bounds match well with
those in \cite{ShaNem04,sha06}, implying that
multi-stage stochastic optimization problems are essentially intractable for $T \ge 5$ and a moderate target accuracy.
Hence, it is reasonable to use the DSA algorithm only for multi-stage stochastic optimization problems with $T$ relatively small and $\epsilon$ relatively large.
However, it is interesting to point out that the DSA algorithm
only needs to go through the scenario tree once and hence its memory requirement increases only linearly with respect to $T$.
Moreover, the development of the complexity bounds of multi-stage stochastic optimization in terms of their dependence on
various problem parameters may help us to further explore the structure of the problems and to identify
special classes of problems possibly admitting faster solution methods.

It is also interesting to compare the DSA method with some other decomposition type algorithms.
As discussed in Section 1, in the sample average approximation approach, we
can apply a few different decomposition methods for solving the
deterministic counterpart of the multi-stage stochastic optimization problem.
These methods need to go through the whole scenario tree many times and hence
it is necessary to store the scenario tree first. One widely used decomposition method
is the stochastic dual dynamic programming (SDDP). Under the stage-wise
independence assumption, SDDP iteratively builds cutting plane models to approximate
the value functions starting from the last stage $T$ until the first stage (backward iteration), and
then generates feasible solutions starting from the first stage to the last stage (forward iteration).
This method is attractive for solving problems with a large number of
stages because its iteration cost only linearly depends on $T$.
On the other hand, as a common drawback for
cutting plane methods, SDDP converges slowly
as the number of decision variables in each stage increases~\cite{Nest04}. Improvement of cutting plane methods, e.g., based on the bundle-level method, however,
can only be applied to two-stage problems only (see \cite{Lan13-1} and references therein).
 Moreover,
 the rate of convergence of SDDP, i.e.,
how many number of forward and backward iterations it will take to achieve
a certain accurate solution, still remains unknown for multi-stage problems with $T \ge 3$,
although its asymptotic convergence has been established for multi-stage linear programming~\cite{Sha11} .

\section{Numerical experiment} \label{sec_num}
Our goal in this section is to report the results from our preliminary numerical experiments conducted
to test the efficiency of the DSA method applied to a class of
multi-stage asset allocation problems.

We consider a classic multistage asset allocation problem due to Dantzig and Infanger \cite{dantzig1993multi} given by
\begin{equation}\label{portfolio}
\begin{array}{l}
\min_{x^0,p^1,q^1} \quad\quad \bbe \left[\min -u(\tsum_{i=1}^{n+1}x_i^1) \quad\quad\quad\quad\quad\quad +\ldots +  \bbe\left[ \min -u(\tsum_{i=1}^{n+1}x_i^T) \right] \right]\\
  \text{s.t.} 0\leq p_i^1\leq \bar p^1, \quad \ \text{s.t.}\  x_i^1 = R_i^1(x_i^0-p_i^1+q_i^1), \quad\quad\quad\quad \text{s.t.}\ x_i^T = R_i^T(x_i^{T-1}-p_i^T+q_i^T),\ \ i=1,\ldots,n,\\
   \quad\   0\leq q_i^1\leq \bar q^1,  \ \ \ \ \ \quad x_{n+1}^1 = x_{n+1}^0+\tsum_{i=1}^n(1-\hat p_i)p_i^1\ \quad\quad x_{n+1}^T = x_{n+1}^{T-1}+\tsum_{i=1}^n(1-\hat p_i)p_i^T,\\
  \quad\quad\quad\quad\quad\quad\quad\quad\quad\quad\quad\quad -\tsum_{i=1}^n(1+\hat q_i)q_i^1,\quad\quad\quad\quad\quad\quad\quad\quad -\tsum_{i=1}^n(1+\hat q_i)q_i^T.\\
   \quad  \tsum_{i=1}^{n+1} x_i^0=w_0,  \ \quad  \ \ 0\leq p_i^2\leq \bar p^2,\ \ \\ 
   \quad\quad   x_i^0\geq 0,  \ \ \ \quad\quad \ \ \ \ \ 0\leq q_i^2\leq \bar q^2, 
\end{array}
\end{equation}
Here $p_i^t$ and $q_i^t$, respectively, denote the amount of asset $i$ that will be sold and purchased in period $t$,
$\hat p_i$ and $\hat q_i$, respectively, denote the transaction costs for selling and purchasing one unit of asset $i$,
and $R_i^t$ represent the factor of random return for asset $i$ from time $t$ to time $t+1$. 
Moreover, the utility function $u(\cdot)$
 describes the investor's risk preference.
In particular, a linear utility function $u(\cdot)$ describes risk neutrality while a concave utility function models risk averseness.
At the initial time period $0$ the decision maker has a total amount of wealth $w_0$ in assets $i = 1, \ldots, n$ and in cash
(indexed as asset $n+1$ for notational convenience).
The dollar values of these initially available assets are denoted by $x_i^0,\ i= 1,\ldots,n+1$.
In each period of time, short-selling of assets and borrowing of cash are allowed when $x_i<0$, but there exist upper bounds $\bar p$ and $\bar q$ on
the selling and buying amount, respectively.
The goal of the decision maker is to maximize the expected utility $\bbe[u(\tsum_{i=1}^{n+1}x_i^T)]$ for the portfolio over $T$ periods of time.

\subsection{Stagewise dependent random return}
Our goal in this subsection is to demonstrate that the DSA method does not require the stage-wise independence assumption
for the random returns. In this set of experiments, we model the correlation between asset returns using a factor model
\begin{equation}\label{return}
R^t = FV^t,
\end{equation}
which relates the asset returns $R^t = (R_1^t,\ldots, R_n^t)'$ to factors $V^t = (v_1^t,\ldots, v_h^t)'$
through a factor matrix $F \in \bbr^{n\times h}$. This factor model will allow us to
consider the stage-wise dependence, e.g., given by
\begin{equation}\label{return_fact}
v_i^t = v_i^{t-1}+ \epsilon_i^t, \ i=1,\ldots, h,
\end{equation}
where $\epsilon_i^t$ denote the independent random variation of the factor $v_i$ in time $t$.
We collected the data of weekly returns for $1,887$ assets from Thomson Reuters Datastream
({\tt http://financial.thomsonreuters.com/}),
and use these data to fit the random return model.
We assume that the investor is risk averse with the utility function $u(\cdot)$ defined as the classic concave quadratic utility function \cite{pedersen2003utility}, i.e., $u(W) = W - bW^2$. 
The value of $W$ and $b$ are chosen according to \cite{pedersen2003utility}.  We generate three instances (Inst 1, Inst 2 and Inst 3) which have a fixed number
of stages $3$, but with different number of assets ($5$, $200$ and $400$).

When implementing the DSA algorithm, we consider every outer mostest loop as one iteration and run the algorithm for $100$ iterations.
For the sake of convenience, we set $N_1 = \ldots=N_T =100$. Note that in order to estimate the function values for an output solution, we
generate $N$ realizations for the random vector $\{\epsilon^t\},t=1,\ldots,T-1$, and form a scenario tree consisting of $N^{T-1}$ random returns $R^{j,t}$ at level $t$ $\forall t=1,\ldots,T, j=1,\ldots,N^{T-1}$
according to \eqref{return} and \eqref{return_fact}. Then we will find a prefixed control policy $\{x^0,\bar p^t, \bar q^t\}, t=1,\ldots,T-1$ based on the
the output of the algorithm, and calculate other state variables according to
\begin{equation}\label{xt_relation}
  x_i^{j,t} = R_{i}^{j,t}(x_i^{t-1}-\bar p_i^{t}+ \bar q_i^{t}), \forall i = 1,\ldots, n.
\end{equation}
In other words, at stage 1, we will get $N$ feasible $\{x^{j,1}\},\forall j=1,\ldots,N$ by \eqref{xt_relation},
and at stage 2, we will get total $N^2$ feasible $\{x^{i,2}|x^{j,1},R^{i,2}\},\forall i=1,\ldots,N^2$ by \eqref{xt_relation} and so on.
Then we estimate the function value by 
\begin{equation}\label{funV}
FV = \tfrac{1}{N} \tsum_{j=1}^N\left[-u(x^{j,1})+ \tfrac{1}{N} \tsum_{i=N(j-1)+1}^{Nj}[-u(x^{i,2})+ \ldots]\right].
\end{equation}
It is worth noting that $FV$ estimates an upper bound on the objective value at $\{x^0\}$.
Nevertheless, our experimental results reported in Table \ref{DSA_results} indicates that DSA does converge for these problems with stagewise dependent return.

\begin{table}
\caption{Problem parameters for stagewise dependent return}
\label{tProblemparameter}
 \begin{tabular}{||c|c|c|c|c|c|c||}
 \hline
 \multicolumn{1}{||c|}{}  & \multicolumn{1}{|c|}{$n$} & \multicolumn{1}{|c|}{$h$} &\multicolumn{1}{|c|}{$w_0$}  & \multicolumn{1}{|c|}{$\bar p = \bar q$} &
   \multicolumn{1}{|c|}{$\hat p = \hat q$} &\multicolumn{1}{|c||}{$T$} \\
 \hline
 \hline
Inst 1& 5 & 3& 3& 0.1 & 0.05 & 3 \\
Inst 2 & 200 & 70 & 500 & 1& 0.05 & 3\\
Inst 3 & 400 & 240& 1,000& 1& 0.05 & 3\\
 \hline
\end{tabular}
%
\centering
\caption{Numerical results for DSA with stagewise dependent return}\label{DSA_results}
\begin{tabular}{|c|c|c|c|c|c|c||}
\hline
\hline
 & \#. of Iter.  & 0 & 10 & 20 &60 & 100  \\ \hline
\multirow{ 2}{*}{Inst 1} & FV & -4.0812 & -4.1047 & -4.1186 & -4.1704 & -4.2352  \\
& Time(s) & 0 & 1.96 & 4.02 & 12.37 & 21.00    \\ \hline
\multirow{ 2}{*}{Inst 2} & FV &-665.79 & -665.99 & -666.13 & -672.38 & -675.80  \\
& Time(s) & 0 & 12.38 & 24.77 & 77.40 & 126.55    \\ \hline
\multirow{ 2}{*}{Inst 3} & FV &-1.3326*e+3 & -1.3334*e+3 & -1.3337*e+3 & -1.3414*e+3 & -1.3493*e+3  \\
& Time(s) & 0 & 56.65 & 114.64 & 339.21 & 565.73    \\ \hline
\end{tabular}

\end{table}

\subsection{Stagewise independent return}
Our goal in the second set of experiments is to compare DSA with SDDP for solving problem~\eqnok{portfolio}.
Since SDDP cannot be directly applied for solving problems with stagewise dependent return,
in order to compare these two algorithms, we assume the random returns are stagewise independent given by
\begin{equation} \label{indep_R}
R^t = \mu+ \epsilon^t, \forall t =1,\ldots,T,
\end{equation}
where $\mu \sim\text{Uniform}[0.8,1.2]$, and $\epsilon^t \sim \text{Normal}(0,\sigma^2)$.
Given starting point $(\bar x_1, \bar x_2, \ldots, \bar x_T)$ and approximation of value function $\mathfrak{Q}_t$ for $t = 1,\ldots,T$,
each iteration of the SDDP algorithm consists of a forward step and a backward step to update the feasible solutions
and the approximate value functions, respectively. Since each stage is independent,
in SDDP, we use the sample size $N_1 = \ldots = N_T=100$ to generate the sample average approximation
problem first. In addition, the number of samples $M$ to compute the upper bound in the forward step at each stage is
set to $20$ (see~\cite{Sha11}).

\begin{table}
\caption{Problem parameters for stagewise independent data}
\label{tProblemparameter2}
 \begin{tabular}{||c|c|c|c|c|c|c||}
 \hline
 \multicolumn{1}{||c|}{}  & \multicolumn{1}{|c|}{$n$}  &\multicolumn{1}{|c|}{$w_0$}  & \multicolumn{1}{|c|}{$\bar p = \bar q$} &
   \multicolumn{1}{|c|}{$\hat p = \hat q$} &\multicolumn{1}{|c||}{$T$} & $\sigma$ \\
 \hline
 \hline
Inst 4& 5 & 3& 0.1 & 0.05 & 3 &0.05 \\
Inst 5 & 200  & 500 & 1& 0.05 & 3 & 0.1\\
Inst 6 & 400 & 1,000& 1& 0.05 & 3 & 0.2\\
Inst 7 & 5  & 3& 0.1 & 0.05& 4 &0.1 \\
Inst 8 & 5  & 3&0.1 & 0.05& 5 &0.1 \\
 \hline
\end{tabular}
\end{table}

We apply both DSA and SDDP to solve a few different problems instances of problem\eqnok{portfolio}
with parameters given in Table~\ref{tProblemparameter2}.
In particular, we consider two subgroups of instances. The first group (Inst 4, Inst 5 and Inst 6) has a fixed number
of stages $3$, but with different number of assets ($5$, $200$ and $400$), while the second group (Inst 4, Inst 7 and Inst 8)
has the same parameter setting except the number of stages changes from $3$ to $4$ or $5$.
Our hypothesis is that the DSA method can scale up with the dimension of the problem (i.e., the number of assets),
while SDDP can handle problems with a larger number of stages.
We first report the estimated
 function values (FV) for the obtained solutions  in Tables~\ref{com_1}, \ref{com_2} and \ref{com_3},
 where the first column
 represents the number of iterations for both algorithms,
 the second and fourth columns represent the estimated function values for DSA and SDDP, respectively,
 and the third and fifth columns are the recorded CPU times for DSA and SDDP, respectively.
 Note that in order to estimate the function values for a generated solution, we
  generate $N$ sequences random variation $\{\epsilon^t\},t=1,\ldots,T$, and get the random returns $R_j^t\in \bbr^n,\forall t=1,\ldots,T, j=1,\ldots,N$
  according to \eqref{indep_R}, then
 we compute feasible solution by \eqref{xt_relation} and estimate the function value by
$FV =  \tfrac{1}{N} \tsum_{j=1}^N\tsum_{t=1}^{T-1} -u(x^{j,t})$
with $N = 1000$. 

We observe from Tables~\ref{com_1}, \ref{com_2} and \ref{com_3} that both DSA and SDDP can obtain solutions with comparable quality,
and that DSA will significantly outperform SDDP in terms of computation time for instances with a small number of stages (e.g., $T =3$). Moreover,
from Table~\ref{com_4}, we can see that for the problem with larger number of stages and a small number of assets ($5$), the computation time for DSA
grows exponentially w.r.t. $T$ while the one for SDDP grows almost linearly.
From these preliminary numerical results, we indeed confirm that DSA can be used to handle multi-stage stochastic optimization
problems with a large number of decision (or state) variables, but a relatively smaller number of stages.
On the other hand, SDDP type algorithms can be used to solve problems with a larger number of stages but smaller number
of decision (or state) variables. These two types of algorithms seem to be complimentary to each other for solving multi-stage stochastic optimization problems.

\begin{table}
\begin{minipage}[b]{0.5\textwidth}
\caption{Comparison for instance 4}\label{com_1}
\label{comparison_1}
 \begin{tabular}{||c||c|c|c|c||}
 \hline
  & \multicolumn{2}{|c|}{DSA} & \multicolumn{2}{|c|}{SDDP} \\
 \hline
  & FV & Time(s) & FV & Time(s) \\
 \hline\hline
 0&-3.8456 & 0 & -3.8492 & 0 \\
 \hline
 10&-3.8737 & 2.37 & -4.1630 & 515.75 \\
 \hline
 20& -3.9156 & 4.40 & -4.1992 & 1154.79\\
 \hline
 60& -4.1121 & 12.26 & -4.2162 & 3023.62 \\
 \hline
 100& -4.1772 & 20.16 & -4.2028 & 5052.51 \\
 \hline
\end{tabular}
\end{minipage}
%
\begin{minipage}[b]{0.5\textwidth}
\caption{Comparison for instance 5}\label{com_2}
\label{comparison_2}
 \begin{tabular}{||c||c|c|c|c||}
 \hline
  & \multicolumn{2}{|c|}{DSA} & \multicolumn{2}{|c|}{SDDP} \\
 \hline
  & FV & Time(s) & FV & Time(s) \\
 \hline\hline
 0&-670.66 & 0 & -670.46 & 0 \\
 \hline
 10&-670.91 & 13.13 & -701.92 & 1801.97 \\
 \hline
 20&-674.80 & 28.08 & -721.09 & 3959.05 \\
 \hline
 60& -713.29 & 89.23 & -721.89 & 17028.26 \\
 \hline
 100& -721.18 & 146.07 & -720.28 & 27826.86 \\
 \hline
\end{tabular}
\end{minipage}
\end{table}

\begin{table}
\caption{Comparison for instance 6}
\label{com_3}
 \begin{tabular}{||c||c|c|c|c||}
 \hline
  & \multicolumn{2}{|c|}{DSA} & \multicolumn{2}{|c|}{SDDP} \\
 \hline
  & FV & Time(s) & FV & Time(s) \\
 \hline\hline
 0& -1325.05 & 0 & -1323.31 & 0 \\
 \hline
 1&-1326.01 & 9.20 & -1327.53 & 347.28 \\
 \hline
 10& -1327.00 & 52.84 & -1427.14 & 3769.16 \\
 \hline
 50& -1412.54 & 295.47 & -1443.81 & 27957.58 \\
 \hline
 100& -1444.51 & 549.55 & -1455.11 & 53772.13 \\
 \hline
\end{tabular}
\end{table}

%
%
%
%

\begin{table}
\centering
\caption{Computing time for instance 7 and 8}
\label{com_4}
 \begin{tabular}{||c||c|c|c|c||}
 \hline
  & \multicolumn{2}{|c|}{T=4} & \multicolumn{2}{|c|}{T=5} \\
 \hline
  & DSA & SDDP & DSA & SDDP \\
 \hline\hline
 1& 26.59 & 61.81 & 1964.09 & 80.73 \\
 \hline
 10& 191.02 & 531.95 & 18617.30 & 693.65 \\
 \hline
 20& 394.18 & 1134.03 & 38875.61 & 1441.82 \\
 \hline
 60& 1214.24 & 3505.01 & 118336.51 & 4452.99 \\
 \hline
 100& 2027.51 & 6035.81 & 191947.11 & 7461.19 \\
 \hline
\end{tabular}
\end{table}

\section{Conclusion} \label{sec-remark}
In this paper, we present a new class of stochastic approximation algorithms, i.e., dynamic stochastic approximation (DSA), for solving multi-stage stochastic optimization problems.
This algorithm is developed by reformulating the optimization problem in each stage as a saddle point problem and then recursively applying
an inexact primal-dual stochastic approximation algorithm to compute an approximate stochastic subgradient of the previous stage.
We establish the convergence of this algorithm by carefully bounding the bias and variance associated with
these approximation errors. For a three-stage stochastic optimization problem, we show that the total number of required scenarios
to find an $\epsilon$-solution is bounded by ${\cal O}(1/\epsilon^4)$ and ${\cal O}(1/\epsilon^2)$, respectively, for general convex and strongly convex cases.
These bounds are essentially not improvable in terms of their dependence on the target accuracy. We also generalize
DSA for solving multi-stage stochastic optimization problems with the number of stages $T > 3$. To the best of our knowledge,
this is the first time that stochastic approximation methods have been developed and their complexity is established for multi-stage stochastic optimization.

From the preliminary numerical results, we can see the DSA method is efficient for solving high dimensional problems with
a relatively smaller number of stages. However, as the number of stages increase, the computing time would increase exponentially even though
it can handle the case when random variable are stage-wise dependent.
Further improvement
on the practical performance of this method should be pursed
along the directions of better
estimating problem parameters especially those related to the size of subgradients and dual multipliers. It would be interesting to study
whether one can estimate these parameters in an online fashion while running these methods, and whether one can
further improve the convergence of DSA in terms of its dependence on these problem parameters, e.g., by using accelerated SA methods and some other
algorithmic schemes.

It is worth noting that there exist a class of alternative approaches based on  linear decision rule models for solving
multi-stage stochastic optimization problems. In these methods we assume that the decisions linearly depend
on the decisions previously made and the realization of random variables that have been observed so far.
Using this approach, one can reformulate a multi-stage stochastic optimization problem into a two-stage problem, and
hence can significantly reduce the computational cost. In comparison with the exact methods we focus on
in this paper, using linear decision rule models can only generate suboptimal solutions
for the original multi-stage stochastic optimization problems in general.

\renewcommand\refname{Reference}

\bibliographystyle{plain}
\bibliography{glan-bib}

\end{document}